\documentclass[11pt]{article}

\usepackage[left=1.2in,right=1.2in]{geometry}

\usepackage[utf8]{inputenc}
\usepackage[T1]{fontenc}
\usepackage{ifthen}
\usepackage{amssymb}
\usepackage{amsthm} 
\usepackage{bbm}
\usepackage{enumitem}
\usepackage{hyperref}

\usepackage{amsmath}
\usepackage{blkarray}

\usepackage[all]{xy}

\usepackage{xcolor}
\definecolor{notecolor}{rgb}{1,0,0}
\newcommand{\note}[1]{{\small\textcolor{notecolor}{(#1)}}}

\usepackage{tikz}
\usepackage{tkz-graph}
\usepackage{tikz-cd}
\usetikzlibrary{arrows.meta, positioning}
\usepackage{xparse}
\usepackage{mathrsfs}
\usepackage{graphicx}
\usepackage{hyperref}

\definecolor{linkcolor}{rgb}{0,0,0.8} 

{
\theoremstyle{plain}
\newtheorem{theorem}{Theorem}[section]
\newtheorem{lemma}[theorem]{Lemma}
\newtheorem{proposition}[theorem]{Proposition}

\newtheorem{example}[theorem]{Example}
\newtheorem{definition}[theorem]{Definition}

\newtheorem*{conjecture*}{Conjecture}
\newtheorem{question}[theorem]{Question}

}

{
\theoremstyle{definition}

\newtheorem{remark}[theorem]{Remark}
}

\newcommand{\sizedescriptor}[2]
{
\ifthenelse{\equal{#1}{0}}{}{
\ifthenelse{\equal{#1}{1}}{\big}{
\ifthenelse{\equal{#1}{2}}{\Big}{
\ifthenelse{\equal{#1}{3}}{\bigg}{
\ifthenelse{\equal{#1}{4}}{\Bigg}{
#2}}}}}
}

\newcommand{\someref}{{\small\textcolor{blue}{[\textbf{ref.}]}} }

\newcommand{\impl}{\Rightarrow}  
\newcommand{\all}[1]{\forall #1 .\,}  
\newcommand{\some}[1]{\exists #1 .\,}  

\newcommand{\df}[1]{\emph{\textbf{#1}}}  
\newcommand{\ism}{\cong}  

\NewDocumentCommand{\set}
	{O{auto} m G{\empty}}
	{\sizedescriptor{#1}{\left}\{ {#2} \ifthenelse{\equal{#3}{}}{}{ \; \sizedescriptor{#1}{\middle}| \; {#3}} \sizedescriptor{#1}{\right}\}}

\newcommand{\pst}{\mathcal{P}}  

\newcommand{\NN}{\mathbb{N}}

\newcommand{\RR}{\mathbb{R}}

\newcommand{\intoo}[3][\RR]{{#1}_{(#2, #3)}}
\newcommand{\intcc}[3][\RR]{{#1}_{[#2, #3]}}
\newcommand{\intoc}[3][\RR]{{#1}_{(#2, #3]}}
\newcommand{\intco}[3][\RR]{{#1}_{[#2, #3)}}

\newcommand{\er}{\overline{\RR}}

\NewDocumentCommand{\oball}  
	{O{\empty} G{\empty} G{\empty}}
	{B_{#1}\ifthenelse{\equal{#2}{}}{}{\!\left(#2, #3\right)}}
\NewDocumentCommand{\cball}  
	{O{\empty} G{\empty} G{\empty}}
	{\overline{B}_{#1}\ifthenelse{\equal{#2}{}}{}{\!\left(#2, #3\right)}}
\NewDocumentCommand{\cth}  
	{O{\empty} G{\empty} G{\empty}}
	{\overline{\mathrm{th}}_{#1}\ifthenelse{\equal{#2}{}}{}{\!\left(#2, #3\right)}}

\newcommand{\ph}{{\text{---}}}  
\newcommand{\parto}{\mathrel{\rightharpoonup}}  
\newcommand{\prj}{\mathrm{pr}}
\NewDocumentCommand{\dimg}  
	{O{\empty} m G{\empty}}
	{{#2}_*\ifthenelse{\equal{#3}{}}{}{\!\sizedescriptor{#1}{\left}( {#3} \sizedescriptor{#1}{\right})}}
\NewDocumentCommand{\pimg}  
	{O{\empty} m G{\empty}}
	{{#2}^*\ifthenelse{\equal{#3}{}}{}{\!\sizedescriptor{#1}{\left}( {#3} \sizedescriptor{#1}{\right})}}

\usepackage{mathtools}
\mathtoolsset{centercolon}

\newcommand{\mg}[1][]{\mathsf{mag}_{#1}}

\newcommand{\fsub}{\mathrm{Fin}}
\newcommand{\ifsub}{\fsub_{+}}
\newcommand{\csub}{\mathrm{Cmp}}
\newcommand{\icsub}{\csub_{+}}

\newcommand{\card}[1]{\mathnormal{\#}{#1}}

\newcommand{\sgn}{\mathrm{sgn}}
\newcommand{\decr}{\mathop{\text{\tiny$\searrow$}}}
\newcommand{\spr}[2]{\langle{#1}, {#2}\rangle}

\newcommand{\dr}[1][N]{\mathrm{Dir}_{#1}}
\newcommand{\dotted}[2]{{\dddot{#1}}({#2})}
\newcommand{\face}{\mathrm{face}}
\newcommand{\fdim}{\mathrm{fdim}}
\newcommand{\sub}{\mathrm{sub}}
\newcommand{\proj}{\mathrm{proj}}

\newcommand{\chg}{\mathrm{chg}}
\newcommand{\cc}{\overline{C}}

\newcommand{\fr}[2][F]{{#1}^{\sqsubset}_{#2}}
\newcommand{\dfr}[2][F]{\mathrm{df}(#1)_{#2}}

\title{Continuity of Magnitude at Finite Subsets of $\ell_1^N$}

\author{
Sara Kali\v{s}nik\thanks{
Pennsylvania State University,
\texttt{skalisnik@psu.edu}}\phantom{x} and 
Davorin Le\v{s}nik\thanks{
University of Ljubljana,
\texttt{davorin.lesnik@fmf.uni-lj.si}}
}

\date{}
\begin{document}

%
%

\maketitle

\begin{abstract}
Magnitude is an isometric invariant of metric spaces introduced by Leinster in 2011. It is nowhere continuous on the Gromov--Hausdorff space of finite metric spaces, however, positive continuity results do exist if we restrict the ambient space. In this paper, we prove that magnitude is continuous at every finite subset $F$ of $\ell_1^N$. We do this by first deriving the weight measure for a finite union of cubes, i.e.\ cubical thickenings of the points in $F$, and then showing that these thickenings converge to the magnitude of the underlying finite set.
\end{abstract}

\tableofcontents

\section{Introduction}
Magnitude is an isometric invariant of finite metric spaces that originally arose in category theory as an extension of Euler characteristic to enriched categories~\cite{leinster2010}. Since then it has been shown to encode, under suitable conditions, numerous geometric quantities of compact metric spaces such as dimension, perimeter, area, and volume~\cite{leinster2010, LW13, Willerton2014, W09, meckes2013}. It has also seen applications ranging 
from ecology~\cite{Solow1994, LC12} to topological data analysis~\cite{O18, OMALLEY2023107396, GH21}. 

In the context of data analysis, continuity of an invariant is central. If we consider the space of isometry classes of finite metric spaces and equip it with the Gromov--Hausdorff metric, magnitude is nowhere continuous~\cite{katsumasa2025magnitudegenericallycontinuousfinite, leinster2010, Roff25}. However, if we restrict to suitable subclasses of metric spaces, we may obtain continuity results. For example, recently Yuki Hiyoshi  considered finite metric spaces admitting a nonnegative weighting and showed within that class, the logarithmic magnitude function is Lipschitz continuous~\cite{hiyoshi2026continuitymagnitudefinitemetric}.

In~\cite{kalisnik2026continuitymagnitudeskewfinite} we have shown that magnitude is continuous at every finite subset of~$\ell_1^N$ which is `skew' --- meaning that all coordinate projections of the subset are injective. To do this, we used the following strategy. We showed that continuity of magnitude at a finite subset $F \subseteq \ell_1^N$ is implied by the property: magnitude of the union of cubes around each point of~$F$ tends to the magnitude of~$F$ as the size of the cubes tends to zero. We showed this property by deriving a formula for the weight measure and the magnitude of the union of cubes which have disjoint projections onto coordinate axis (which happens for sufficiently small cubes around the points of a skew subset).

Since skew finite subsets form a dense open subspace of the space of all finite subsets of~$\ell_1^N$ (in the Hausdorff distance), this showed that magnitude is continuous at ``almost all'' finite subsets of~$\ell_1^N$. In this paper, we generalize our methods so that we can drop the condition of `skewness', i.e.\ we show that magnitude is continuous at \emph{every} finite subset of~$\ell_1^N$.

The strategy in this paper is similar as in~\cite{kalisnik2026continuitymagnitudeskewfinite}. The first step towards the proof is to derive the weight measure of a finite union of disjoint cubes in $\ell_1^N$: for $p \in \ell_1^N$ and $r \geq 0$ we write $\cc_p(r) := {\prod_{k \in \intcc[\NN]{1}{N}} [p_k - r, p_k + r]}$ and, for a finite set $F \subseteq \ell_1^N$, we set $\cc_F(r) := \bigcup_{p \in F} \cc_p(r)$. For small positive~$r$, $\cc_F(r)$ admits a weight measure which can be expressed as a certain linear combination of the weight measures of the faces of cubes, see Theorem~\ref{theorem:weight-measure-of-union-of-cubes}. As cube faces are products of intervals/points, their weight measures are known.

Using this cube weight measure formula and Lemma~\ref{lemma:weight-limits}, we deduce that
\[
\lim_{r \decr 0} \mg\big(\cc_F(r)\big) = \mg(F)
\qquad
\text{for every finite $F \subseteq \ell_1^N$}.
\]
It follows that magnitude, restricted to~$\ell_1^N$, is continuous at every finite subset (Theorem~\ref{theorem:finite-magnitude-continuity}). 

While the two papers share this broad strategy, we stress certain important differences.
\begin{itemize}
\item
In~\cite{kalisnik2026continuitymagnitudeskewfinite}, we gave the weight measure of a union of cubes around points of a finite skew subspace in terms of Lebesgue and Dirac measures. While this could in principle be done also for general finite subspaces (cf.~Example~\ref{example:pair}), it would be very difficult to prove the above results in this way. Both the theoretical derivations and practical computations simplify significantly when the cube union weight measure is given in terms of the weight measures of the cube faces. This is because of the relatively simple formula, given in Lemma~\ref{lemma:face-weight-measure-integration}.
\item
The linear system, giving us the coefficients in formula for the weight measure of a cube union, is much more complicated in the general case. In~\cite{kalisnik2026continuitymagnitudeskewfinite}, we could give it in terms of `corners'. These needed to be generalized in two different ways, so that we obtained `fragments' (Definition~\ref{definition:fragments}) and `dotted fragments' (Definition~\ref{definition:dotted-fragments}).
\item
The crucial part of the magnitude continuity argument is that the coefficient matrix of the linear system, determining the coefficients in the formula for the cube union weight measure, is invertible at (and hence in some neighborhood of) $r = 0$. In~\cite{kalisnik2026continuitymagnitudeskewfinite}, we did this in Lemma~6.1, essentially by explicitly giving the unique solution of the system. This approach did not lend well to generalization. Instead, we needed two complicated induction, see Lemma~\ref{lemma:fragment-system-solvable-at-zero} and its proof.
\end{itemize}

\subsection{Notation and Conventions}\label{sub:notation}

We denote the set of natural numbers by~$\NN$. We treat~$0$ as a natural number, so $\NN = \set{0, 1, 2, 3,\ldots}$.

We denote the set of real numbers by~$\RR$, and the set of extended real numbers (with infinities included) by~$\er$. That is, $\er = \RR \cup \set{-\infty, \infty}$.

Subsets, given by a relation, are denoted by that relation in the index; for example, $\RR_{\geq 0}$ is the set of non-negative real numbers. Intervals between two numbers are denoted by these two numbers in brackets and in the index. Round, or open, brackets $(\ )$ denote the absence of the boundary in the set, and square, or closed, brackets $[\ ]$ its presence. For example, $\intcc{0}{1} = \set{x \in \RR}{0 \leq x \leq 1}$ is the usual closed unit interval, and $\intco[\NN]{5}{10} = \set{n \in \NN}{5 \leq n < 10} = \set{5, 6, 7, 8, 9}$.

For any set~$A$, we denote its cardinality by~$\card{A}$.

The powerset of a set~$A$ (the set of all subsets of~$A$) is denoted by~$\pst(A)$. We denote the set of finite subsets of~$A$ by~$\fsub(A)$, and the set of non-empty finite subsets by~$\ifsub(A)$. If $A$ is also equipped with a topology, then $\csub(A)$ denotes the set of compact subspaces of~$A$, and $\icsub(A)$ the set of non-empty compact subspaces of~$A$.

A function $f$ mapping from a set~$A$ to a set~$B$ is denoted as $f\colon A \to B$. If $f$ is merely a partial map (not necessarily defined on the whole~$A$), we write this as $f\colon A \parto B$.


For any $N \in \NN$ and $k \in \intcc[\NN]{1}{N}$, we use $\prj_k\colon \RR^N \to \RR$ to denote the $k$-th coordinate projection map, i.e.\ $\prj_k(x_1, \ldots, x_N) = x_k$.

We use $\spr{a}{b}$ to denote the usual scalar product (or dot product, or inner product) of vectors $a, b \in \RR^N$, i.e.\ $\spr{a}{b} = \sum_{k \in \intcc[\NN]{1}{N}} a_k\:\!b_k$.

For any $p \in \er_{\geq 1}$ and $N \in \NN$, let $\ell_p^N$ denote the Banach space~$\RR^N$, equipped with the $p$-norm (and the induced $p$-metric), and let $\ell_p$ denote the infinite-dimensional version of this space. Moreover, $L_1$ is the shorthand for $L_1\big(\intcc{0}{1}, \RR\big)$, i.e.\ the space of (equivalence classes of) measurable functions $\intcc{0}{1} \to \RR$, equipped with the usual integral $1$-norm. We will use $d_1$ to denote the $1$-metric and $\|\ \|_1$ the $1$-norm on these spaces.


Given a metric space~$M$ with a metric~$d$, the open ball in~$M$ with the center $x \in M$ and radius $r$ is denoted by $\oball[M]{x}{r}$, likewise for the closed ball $\cball[M]{x}{r}$. We shorten the notation for a ball in $p$-metric to $\oball[p]{x}{r}$, resp.~$\cball[p]{x}{r}$. In particular, we have
\[\cball[\infty]{x}{r} \ =\!\!\prod_{k \in \intcc[\NN]{1}{N}}\!\!\intcc{x_k - r}{x_k + r},\]
and we denote this (closed) cube by $\cc_x(r)$.


We write $\displaystyle{\lim_{x \decr a} f(x)}$ for the right-sided limit (i.e.\ the limit when $x$ approaches~$a$ from above).

Write $1_N$ (resp.~$0_N$) for the vector in~$\RR^N$ with all components~$1$ (resp.~$0$). More generally, let $1_F$ (resp.~$0_F$) denote the ``vector'' whose components are indexed by a finite set~$F$ (formally, a function $F \to \RR$), and are all equal to~$1$ (resp.~$0$).

For any matrix~$A$, we write $\mathrm{sum}(A)$ for the sum of all its entries. This includes matrices with a single column, i.e.\ $\mathrm{sum}(v)$ is the sum of all components of a vector~$v$.

\section{Magnitude}

In this section, we provide the background information on magnitude and state some standard results concerning the magnitude of finite and compact metric spaces~\cite{leinster2010, meckes2013}.

We first define magnitude of finite metric spaces~\cite{leinster2010}.
\begin{definition}\label{definition:finite-metric-space-magnitude}
Let $F$ be a finite set, equipped with a metric~$d$. The matrix $Z_F \in \RR^{F \times F}$, given by $Z_F := \big(e^{-d(p, q)}\big)_{p, q \in F}$, is called the \df{similarity matrix} of~$F$. A vector $w \in \RR^F$ is a \df{weighting} for $Z_F$ when $Z_F \cdot w = 1_F$, where $1_F \in \RR^F$ is the vector with all components equal to~$1$. If a weighting~$w$ for $Z_F$ exists, the \df{magnitude} of the metric space~$F$ is defined to be
\[\mg(F) := \mathrm{sum}(w),\]
i.e.\ the sum of all components of the weighting~$w$ (this sum is independent of the choice of the weighting~\cite{leinster2010}).

If $Z_F$ is invertible, there exists a unique weighting $w = Z_F^{-1} \cdot 1_{F}$ and in that case the magnitude of~$F$ equals the sum of all entries of the matrix~$Z_F^{-1}$.
\end{definition}

Let $\fsub(M)$ denote the set of all finite subsets of~$M$ equipped with the Hausdorff distance~$d_H$. This is an extended metric as the empty subset is at infinite distance from the others. If we restrict ourselves to the set of all \emph{non-empty} subsets of~$M$, which we denote by~$\ifsub(M)$, the Hausdorff distance becomes a metric. 

Magnitude induces a partial function $\mg[\fsub(M)]\colon \fsub(M) \parto \RR$. It is continuous if and only if its restriction to~$\ifsub(M)$, denoted $\mg[\ifsub(M)]$, is continuous. This holds since $\emptyset$ is an isolated point in~$\fsub(M)$.

Not every similarity matrix has a weighting, so the magnitude is not defined for every finite metric space. One class of finite metric spaces for which the existence of magnitude is guaranteed are positive definite finite metric spaces, i.e.\ finite metric spaces that have positive definite similarity matrices. Sylvester's criterion then ensures that its every subspace has this property as well. The following is a generalization of this property to metric spaces that are not necessarily finite.

\begin{definition}
A metric space~$M$ is \df{positive definite} when its every finite subspace has a positive definite similarity matrix.
\end{definition}

Examples of positive definite metric spaces, relevant to this paper, are spaces $\ell_1^N$ for all dimensions $N \in \NN$, and all of their subspaces.

If $M$ is positive definite, the maps $\mg[\fsub(M)]$ and $\mg[\ifsub(M)]$ are total. Magnitude has other additional desirable properties~\cite{meckes2013}. For example, it is \df{inclusion-monotone}, i.e.\ for all $F', F'' \in \fsub(M)$, if $F' \subseteq F''$, then $\mg(F') \leq \mg(F'')$~\cite[
Corollary 2.4.4]{leinster2010}.

The definition of magnitude may be extended to compact metric spaces in various ways~\cite{leinster2010, LW13, Willerton2014, W09, meckes2013} and these ways have been shown to be equivalent for positive definite spaces, but not general ones~\cite{meckes2013}.  The approach we use in our paper is via weight measures~\cite[Section 2.1]{Willerton2014}.

\begin{definition}
Let $(K, d)$ be a compact positive definite metric space. 
A finite signed Borel measure $\omega$ on $K$ is a \df{weight measure} 
when it satisfies
\[
\int_K e^{-d(x, a)} \, d\omega(x) = 1
\quad \text{for all } a \in K.
\]
If a compact metric space $K$ admits a weight measure $\omega$, 
then its magnitude is given by
\[
\mg(K) = \omega(K) = \int_K d\omega(x).
\]
\end{definition}

This definition allows us to define $\mg\colon \csub(M) \parto \RR$ (or $\mg[\csub(M)]$ for short) where $\csub(M)$ stands for the set of all compact subspaces of~$M$. We equip $\csub(M)$ with the Hausdorff distance which becomes a metric when we restrict ourselves to the set of all \emph{non-empty} compact subspaces of~$M$, $\icsub(M)$. Again, because $\emptyset$ is an isolated point~$\mg[\csub(M)]$  is continuous precisely when~$\mg[\icsub(M)]$ is.

For positive definite compact metric spaces, weight measures are unique. 
\begin{proposition}\label{proposition:uniqueness-of-weight-measure}
A positive definite compact metric space~$K$ can have at most one weight measure. This includes $K \in \csub(\ell_1^N)$ for all $N \in \NN$.
\end{proposition}

\begin{proof}
See Corollary~1 and the subsequent discussion at the beginning of Subsection~2.3 in~\cite{bouafia2026magnitudediversitytrees}.
\end{proof}

Because of this uniqueness, we will use $\omega_K$ to denote the weight measure of a positive definite compact metric space~$K$ whenever this measure exists.

\section{Cubes}\label{section:cubes}

In this section we introduce the notation and discuss the basic properties of cubes in~$\ell_1^N$.

For any $p \in \ell_1^N$ and $r \in \RR_{\geq 0}$, let
\[\cc_p(r) := \prod_{k \in \intcc[\NN]{1}{N}} \intcc{p_k - r}{p_k + r}.\]
That is, $\cc_p(r)$ is (at least when $r > 0$) an $N$-dimensional closed cube whose edges are aligned with coordinate axes. In this paper, these are the only kind of cubes we are interested in, so when we use the term `cube', this is the kind of cube we mean. The point~$p$ is called the \df{center} of the cube and $r$ its \df{radius}.

For any $F \in \fsub(\ell_1^N)$ and $r \in \RR_{\geq 0}$, let $\cc_F(r) := \bigcup_{p \in F} \cc_p(r)$ (the union of cubes centered at points in~$F$).

Cubes are in particular polytopes, and we may speak of their faces. To this end, the following notation will be convenient.

Let
\begin{align*}
L_{-1} &:= -1, & R_{-1} &:= -1, \\
L_0 &:= -1, & R_0 &:= 1, \\
L_1 &:= 1, & R_1 &:= 1.
\end{align*}
Then, for every $p \in \ell_1^N$, $r \in \RR_{\geq 0}$ and $s \in \set{-1, 0, 1}^N$, define
\[\face_{p, s}(r) := \prod_{k \in \intcc[\NN]{1}{N}} \intcc{p_k + r L_{s_k}}{p_k + r R_{s_k}}.\]
That is, $\face_{p, s}(r)$ is the face of the cube $\cc_p(r)$ which has the center in $p + r s$. Its dimension is
\[\fdim_s := \sum_{k \in \intcc[\NN]{1}{N}} \big(1 - |s_k|\big) = N - \|s\|_1,\]
i.e.\ the number of zeroes in~$s$.

Note that $\face_{p, 0_N}(r)$ is the whole cube $\cc_p(r)$. Denote the other directions by
\[\dr := \set{-1, 0, 1}^N \setminus \set{0}^N;\]
then faces $\face_{p, s}(r)$ for $s \in \dr$ are precisely the proper faces of the cube $\cc_p(r)$. We will also want notation for the set of centers of the proper faces in a finite union of cubes $\cc_F(r)$:
\[\dotted{F}{r} := \set{p + r s}{p \in F \land s \in \dr}.\]

\begin{remark}
The illustrative text above only really makes sense for $r > 0$ (for $r = 0$, the sets $\cc_p(r)$ and $\face_{p, s}(r)$ are singletons). Nevertheless, we want the formal definitions to be exactly as they are given above, including for $r = 0$ (yes, even the ``face dimension'' $\fdim_s$), as we need the system of linear equations, given in Definition~\ref{definition:fragment-system} below, to be defined also for $r = 0$, and its coefficients to be continuous (with respect to~$r$) there.
\end{remark}

For any $s \in \set{-1, 0, 1}^N$, let
\[\sub_s := \set[1]{t \in \set{-1, 0, 1}^N}{\all{k \in \intcc[\NN]{1}{N}}{s_k \neq 0 \impl t_k = s_k}}.\]
Intuitively, $\sub_s$ consists of indices of all subfaces of the face, indexed by~$t$. More precisely, the following holds.

\begin{proposition}
Let $N \in \NN$. The following statements are equivalent for all $s, t \in \set{-1, 0, 1}^N$.
\begin{enumerate}[label=(\roman*)]
\item
$t \in \sub_s$
\item
$\all{p \in \ell_1^N}\all{r \in \RR_{> 0}}{\face_{p, t} \subseteq \face_{p, s}}$
\item
$\some{p \in \ell_1^N}\some{r \in \RR_{> 0}}{\face_{p, t} \subseteq \face_{p, s}}$
\end{enumerate}
\end{proposition}

\begin{proof}
Straightforward from the definitions.
\end{proof}

Cubes in~$\ell_1^N$ turn out to be useful because it is possible to project onto them (and their faces) in a nice way.

For any $p \in \ell_1^N$, $r \in \RR_{\geq 0}$ and $s \in \set{-1, 0, 1}^N$, define the map $\proj_{p, s, r}\colon \ell_1^N \to \ell_1^N$ by
\[\proj_{p, s, r}(x) := \Big(\min\set[1]{\max\set{x_k, p_k + r L_{s_k}}, p_k + r R_{s_k}}\Big)_{k \in \intcc[\NN]{1}{N}}.\]
Clearly we always have $\proj_{p, s, r}(x) \in \face_{p, s}(r)$; in fact $\proj_{p, s, r}$ is the projection onto $\face_{p, s}(r)$. More precisely, the following holds.

\begin{proposition}\label{proposition:face-projection}
Let $N \in \NN$, $p \in \ell_1^N$, $r \in \RR_{\geq 0}$, $s \in \set{-1, 0, 1}^N$. Then, for every $x \in \ell_1^N$ and every $y \in \face_{p, s}(r)$, we have
\[d_1(x, y) = d_1\big(x, \proj_{p, s, r}(x)\big) + d_1\big(\proj_{p, s, r}(x), y\big).\]
Hence $d_1\big(x, \face_{p, s}(r)\big) = d_1\big(x, \proj_{p, s, r}(x)\big)$.
\end{proposition}

\begin{proof}
Consider first the case $N = 1$. If $x \leq p - r$, then
\begin{gather*}
d_1\big(x, \proj_{p, s, r}(x)\big) + d_1\big(\proj_{p, s, r}(x), y\big) = d_1(x, p - r) + d_1(p - r, y) \\
= p - r - x + y - (p - r) = y - x = d_1(x, y).
\end{gather*}
Similarly, if $x \geq p + r$, then
\begin{gather*}
d_1\big(x, \proj_{p, s, r}(x)\big) + d_1\big(\proj_{p, s, r}(x), y\big) = d_1(x, p + r) + d_1(p + r, y) \\
= x - (p + r) + p + r - y = x - y = d_1(x, y).
\end{gather*}
Finally, if $p - r \leq x \leq p + r$, then
\begin{gather*}
d_1\big(x, \proj_{p, s, r}(x)\big) + d_1\big(\proj_{p, s, r}(x), y\big) = d_1(x, x) + d_1(x, y) = 0 + d_1(x, y) = d_1(x, y).
\end{gather*}
We get the case for general~$N$ by just summing this result over all components.

It follows that
\begin{gather*}
d_1\big(x, \face_{p, s}(r)\big) = \inf\set[1]{d_1(x, y)}{y \in \face_{p, s}(r)} \\
= \inf\set[1]{d_1\big(x, \proj_{p, s, r}(x)\big) + d_1\big(\proj_{p, s, r}(x), y\big)}{y \in \face_{p, s}(r)} \\
= d_1\big(x, \proj_{p, s, r}(x)\big) + \inf\set[1]{d_1\big(\proj_{p, s, r}(x), y\big)}{y \in \face_{p, s}(r)} \\
= d_1\big(x, \proj_{p, s, r}(x)\big)
\end{gather*}
since $d_1\big(\proj_{p, s, r}(x), y\big) \geq 0$ and specifically for $y = \proj_{p, s, r}(x)$ we have $d_1\big(\proj_{p, s, r}(x), y\big) = 0$.
\end{proof}

Consider what happens when we project from a cube to one of its faces.

\begin{proposition}\label{proposition:face-projection-from-cube-itself}
Let $N \in \NN$, $p \in \ell_1^N$, $r \in \RR_{\geq 0}$ and $s \in \set{-1, 0, 1}^N$. Then, for every $t \in \intcc{-1}{1}^N$,
\[d_1\big(p + r t, \face_{p, s}(r)\big) = r \big(\|s\|_1 - \spr{s}{t}\big).\]
\end{proposition}

\begin{proof}
By the following calculation (for the first equality, use Proposition~\ref{proposition:face-projection}):
\begin{gather*}
d_1\big(p + r t, \face_{p, s}(r)\big) = d_1\big(p + r t, \proj_{p, s, r}(p + r t)\big) \\
= \sum_{k \in \intcc[\NN]{1}{N}} \Big|p_k + r t_k - \big(\proj_{p, s, r}(p + r t)\big)_k\Big| = \sum_{k \in \intcc[\NN]{1}{N}} \begin{cases} 0 & \text{if $s_k = 0$}, \\ r (1 - t_k) & \text{if $s_k = 1$}, \\ r (1 + t_k) & \text{if $s_k = -1$} \end{cases} \\
= \sum_{k \in \intcc[\NN]{1}{N}} r \big(|s_k| - s_k t_k\big) = r \big(\|s\|_1 - \spr{s}{t}\big).
\end{gather*}
\end{proof}

Let us now consider the weight measure of cubes. Given a measurable subset $S \subseteq \ell_1^N$ and $D \in \NN$, we use $\lambda_S^D$ to denote the $D$-dimensional Lebesgue measure on~$S$. In particular, the $0$-dimensional Lebesgue measure is the counting measure.

In this paper, we establish the convention that any measure, defined on a subset of~$\ell_1^N$, is understood to also be a measure on the whole~$\ell_1^N$, in the sense that everything outside of the subset has measure~$0$. In particular, for any measurable $A \subseteq \ell_1^N$, we have $\lambda_S^D(A) = \lambda_S^D(A \cap S)$. Following this convention, for any $p \in \ell_1^N$ we have $\lambda_{\set{p}}^0 = \delta_p$ where $\delta_p$ is the Dirac measure, centered on~$x$. More generally, for any $F \in \fsub(\ell_1^N)$ we have $\lambda_F^0 = \sum_{p \in F} \delta_p$.

Given $a, b \in \RR$ with $a \leq b$, it is known by~\cite[
Theorem 2]{Willerton2014} that the weight measure of the interval $\intcc{a}{b}$ is
\[\omega_{\intcc{a}{b}} = \tfrac{1}{2} \big(\delta_a + \lambda_{\intcc{a}{b}}^1 + \delta_b\big),\]
from which we get $\mg(\intcc{a}{b}) = 1 + \tfrac{b - a}{2}$. Since the weight measure (resp.~magnitude) of a $1$-product is the product of weight measures (resp.~magnitudes) of the factors (see~\cite[Proposition 5.3.7.]{LM17}), we immediately get the weight measure of a cube $\cc_p(r)$:
\[\omega_{\cc_p(r)} = \prod_{k \in \intcc[\NN]{1}{N}}\!\!\!\!\tfrac{1}{2} \big(\delta_{p_k - r} + \lambda_{\intcc{p_k - r}{p_k + r}}^1 + \delta_{p_k + r}\big) = \tfrac{1}{2^N}\!\!\!\!\!\!\sum_{s \in \set{-1, 0, 1}^N}\!\!\!\!\!\!\lambda_{\face_{p, s}(r)}^{\fdim_s},\]
and the magnitude $\mg(\cc_p(r)) = (1 + r)^N$.

More generally, faces of cubes are also product sets, so their weight measure and magnitude can also be calculated in this way. Taking into account that the weight measure of a singleton is the Dirac measure, centered at that singleton, we get for any $p \in \ell_1^N$, $r \in \RR_{> 0}$ and $s \in \set{-1, 0, 1}^N$
\[\omega_{\face_{p, s}(r)} = \tfrac{1}{2^{\fdim_s}} \sum_{t \in \sub_s} \lambda_{\face_{p, t}(r)}^{\fdim_t},\]
and $\mg(\face_{p, s}(r)) = (1 + r)^{\fdim_s}$.

By definition of a weight measure, we have $\int_{\face_{p, s}(r)} e^{-d_1(x, q)} \, d\omega_{\face_{p, s}(r)}(x) = 1$ for every $q \in \face_{p, s}(r)$. However, the fact that we can project nicely onto cube faces allows us to give a simple formula for this integral also for general $q \in \ell_1^N$. This will be useful later in the derivation of the formula for the weight measure of a union of cubes.

\begin{lemma}\label{lemma:face-weight-measure-integration}
Let $N \in \NN$, $p \in \ell_1^N$, $r \in \RR_{\geq 0}$ and $s \in \set{-1, 0, 1}^N$. Then, for every $q \in \ell_1^N$,
\[\int_{\face_{p, s}(r)} e^{-d_1(x, q)} \, d\omega_{\face_{p, s}(r)}(x) = e^{-d_1(\face_{p, s}(r), q)}.\]
\end{lemma}

\begin{proof}
Using Proposition~\ref{proposition:face-projection} and the fact that $\omega_{\face_{p, s}(r)}$ is the weight measure of~$\face_{p, s}(r)$, we get
\begin{gather*}
\int_{\face_{p, s}(r)} e^{-d_1(x, q)} \, d\omega_{\face_{p, s}(r)}(x) \\
= \int_{\face_{p, s}(r)} e^{-(d_1(q, \face_{p, s}(r)) + d_1(x, \proj_{p, s, r}(q)))} \, d\omega_{\face_{p, s}(r)}(x) \\
= e^{-d_1(q, \face_{p, s}(r))} \cdot \int_{\face_{p, s}(r)} e^{-d_1(x, \proj_{p, s, r}(q))} \, d\omega_{\face_{p, s}(r)}(x) \\
= e^{-d_1(q, \face_{p, s}(r))} \cdot 1 = e^{-d_1(q, \face_{p, s}(r))}.
\end{gather*}
\end{proof}

The formula for the weight measure of a union of cubes that we give later in Theorem~\ref{theorem:weight-measure-of-union-of-cubes} requires that projections of cubes onto coordinate axes are the same or disjoint. This is true of $\cc_F(r)$ for any $F \in \fsub(\ell_1^N)$ and sufficiently small $r \in \RR_{\geq 0}$. Just how small is given by the following quantity.

\begin{definition}
For any $N \in \NN$, let the \df{coordinate half-gap} of $F \in \fsub(\ell_1^N)$ be defined as
\[\chg(F) :=
\begin{cases}
\infty & \text{if $\card{F} \leq 1$}, \\
\frac{1}{2} \min\big(\set{|p_k - q_k|}{p, q \in F \land k \in \intcc[\NN]{1}{N}} \setminus \set{0}\big) & \text{if $\card{F} \geq 2$}.
\end{cases}
\]
\end{definition}

Observe:
\begin{itemize}
\item
$\chg(F) > 0$ for every $F \in \fsub(\ell_1^N)$,
\item
for every $r \in \RR_{\geq 0}$ we have $r < \chg(F)$ if and only if for all $p, q \in F$ and $k \in \intcc[\NN]{1}{N}$, the images of $\cc_p(r)$ and $\cc_q(r)$ under the $k$-th coordinate projection are either the same or disjoint. In particular, this implies that cubes in~$\cc_F(r)$ are pairwise disjoint.
\end{itemize}

The following trivial lemma is nonetheless useful to be able to refer to.

\begin{lemma}\label{lemma:coordinate-half-gap}
The following holds for all $N \in \NN$, $F \in \ifsub(\ell_1^N)$, $p, q \in F$, $r \in \intco{0}{\chg(F)}$ and $k \in \intcc[\NN]{1}{N}$.
\begin{enumerate}
\item
If $p_k > q_k$, then $p_k - r > q_k + r$.
\item
If $p_k < q_k$, then $p_k + r < q_k - r$.
\end{enumerate}
\end{lemma}

\begin{proof}
Immediate from the definition of the coordinate half-gap.
\end{proof}

\section{Fragments}

In this section we consider a specific way to partition a finite set, as well as the set of centers of proper faces of cubes around a finite set. We will need this for the definition of a certain system of linear equations in Section~\ref{section:fragment-system}.

\begin{definition}\label{definition:fragments}
Given $N \in \NN$, $F \in \ifsub(\ell_1^N)$, $q \in \ell_1^N$ and $u \in \set{-1, 0, 1}^N$, we define the \df{$u$-fragment of~$F$ at~$q$} to be the set
\[\fr{q, u} := \set{p \in F}{\all{k \in \intcc[\NN]{1}{N}}{\sgn(p_k - q_k) = u_k}}.\]
\end{definition}

For visualization, imagine putting the origin of a coordinate system to the point~$q$, then add all the coordinate semi-axes, quadrants, and the higher dimensional equivalents. These cut the whole space~$\ell_1^N$ into pieces, and a fragment of~$F$ is the part of~$F$ contained in one such piece.

\begin{example}
Consider $F = \set{a, b, c, d} \subseteq \ell_1^2$, as given by the following picture. For each point in $F$ we can cut~$\ell_1^N$ into the following: the point itself, 4 semi-axes and 4 (open) quadrants.
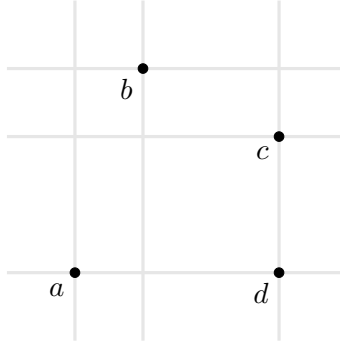
\begin{center}
\begin{tikzpicture}[scale = 0.9]
\draw[black!10, very thick] (-1, 0) -- (4, 0);
\draw[black!10, very thick] (-1, 3) -- (4, 3);
\draw[black!10, very thick] (-1, 2) -- (4, 2);
\draw[black!10, very thick] (0, -1) -- (0, 4);
\draw[black!10, very thick] (1, -1) -- (1, 4);
\draw[black!10, very thick] (3, -1) -- (3, 4);
\filldraw (0, 0) circle (2pt) node [below left] {$a$};
\filldraw (1, 3) circle (2pt) node [below left] {$b$};
\filldraw (3, 2) circle (2pt) node [below left] {$c$};
\filldraw (3, 0) circle (2pt) node [below left] {$d$};
\end{tikzpicture}
\end{center}
Out of $3^N \cdot \card{F} = 3^2 \cdot 4 = 36$ total fragments of~$F$ at a point in~$F$, $22$ are empty. The remaining $14$~fragments are the following.
\begin{align*}
&\fr{a, (0, 0)} = \set{a} & &\fr{b, (0, 0)} = \set{b} & &\fr{c, (0, 0)} = \set{c} & &\fr{d, (0, 0)} = \set{d} \\
&\fr{a, (1, 1)} = \set{b, c} & &\fr{b, (-1, -1)} = \set{a} & &\fr{c, (-1, -1)} = \set{a} & &\fr{d, (-1, 0)} = \set{a} \\
&\fr{a, (1, 0)} = \set{d} & &\fr{b, (1, -1)} = \set{c, d} & &\fr{c, (-1, 1)} = \set{b} & &\fr{d, (-1, 1)} = \set{b} \\
&&&& &\fr{c, (0, -1)} = \set{d} & &\fr{d, (0, 1)} = \set{c}
\end{align*}
We get infinitely many further fragments when we take $q$ outside of~$F$. For example, if $q$ is in the interior of the square just left of the points $c$ and~$d$, we have $\fr{q, (-1, -1)} = \set{a}$, $\fr{q, (-1, 1)} = \set{b}$, $\fr{q, (1, 1)} = \set{c}$, $\fr{q, (1, -1)} = \set{d}$, whereas the $(0, 0)$-, $(0, -1)$-, $(0, 1)$-, $(-1, 0)$-, and $(1, 0)$-fragments of~$F$ at~$q$ are empty.
\end{example}

\begin{lemma}\label{lemma:fragments}
Let $N \in \NN$ and $F \in \ifsub(\ell_1^N)$.
\begin{enumerate}
\item\label{lemma:fragments:partition}
For any $q \in F$, the fragments $\set{\fr{q, u}}{u \in \dr}$ form a partition of $F \setminus \set{q}$. That is, $F \setminus \set{q} = \bigcup_{u \in \dr} \fr{q, u}$, and this is a disjoint union.
\item\label{lemma:fragments:projection-formula}
For all $q \in F$, $u \in \dr$, $p \in \fr{q, u}$, $t \in \intcc{-1}{1}^N$ and $r \in \intco{0}{\chg(F)}$, we have
\[d_1\big(\cc_p(r), q + r u\big) + d_1\big(\face_{q, u}(r), q + r t\big) = d_1\big(\cc_p(r), q + r t\big).\]
\end{enumerate}
\end{lemma}

\begin{proof}
\
\begin{enumerate}
\item
Clear from the definition.
\item
It follows from Lemma~\ref{lemma:coordinate-half-gap} that, for each $k \in \intcc[\NN]{1}{N}$, exactly one of the following happens: $p_k = q_k$ (when $u_k = 0$) or $p_k - r > q_k + r$ (when $u_k = 1$) or $p_k + r < q_k - r$ (when $u_k = -1$). From this and Proposition~\ref{proposition:face-projection} we get
\begin{gather*}
d_1\big(\cc_p(r), q + r t\big) = d_1\big(\face_{p, 0_N}(r), q + r t\big) = d_1\big(\proj_{p, 0_N, r}(q + r t), q + r t\big) \\
= \sum_{k \in \intcc[\NN]{1}{N}} \Big|\big(\proj_{p, 0_N, r}(q + r t)\big)_k - (q_k + r t_k)\Big| = \sum_{k \in \intcc[\NN]{1}{N}} \begin{cases} 0 & \text{if $u_k = 0$}, \\ p_k - q_k - r (1 + t_k) & \text{if $u_k = 1$}, \\ q_k - p_k - r (1 - t_k) & \text{if $u_k = -1$} \end{cases} \\
= \sum_{k \in \intcc[\NN]{1}{N}} \Big(u_k (p_k - q_k) - r \big(|u_k| + u_k t_k)\big)\Big) = \spr{u}{p - q} - r \big(\|u\|_1 + \spr{u}{t}\big).
\end{gather*}
As this holds for all $t \in \intcc{-1}{1}^N$, we may apply it to $t = u$, which gives us
\[d_1\big(\cc_p(r), q + r u\big) = \spr{u}{p - q} - r \big(\|u\|_1 + \spr{u}{u}\big) = \spr{u}{p - q} - 2 r \|u\|_1.\]
By Proposition~\ref{proposition:face-projection-from-cube-itself} we also have $d_1\big(\face_{q, u}(r), q + r t\big) = r \big(\|u\|_1 - \spr{u}{t}\big)$. Putting all this together gives us
\begin{gather*}
d_1\big(\cc_p(r), q + r u\big) + d_1\big(\face_{q, u}(r), q + r t\big) = \spr{u}{p - q} - 2 r \|u\|_1 + r \big(\|u\|_1 - \spr{u}{t}\big) \\
= \spr{u}{p - q} - r \big(\|u\|_1 + \spr{u}{t}\big) = d_1\big(\cc_p(r), q + r t\big).
\end{gather*}
\end{enumerate}
\end{proof}

Next, we consider how to partition the collection of centers of proper faces of a finite union of cubes --- or, to be more precise, the set by which we index them.

\begin{definition}\label{definition:dotted-fragments}
Given $N \in \NN$, $F \in \ifsub(\ell_1^N)$, $q \in \ell_1^N$ and $u \in \set{-1, 0, 1}^N$, we define the \df{dotted $u$-fragment of~$F$ at~$q$} to be the set
\[\dfr{q, u} := \set[2]{(p, s) \in F \times \dr}{p + \tfrac{\chg(F)}{2} s \in \fr[\dotted{F}{\tfrac{\chg(F)}{2}}]{q, u}}.\]
\end{definition}

Note that a dotted fragment could equivalently be defined as
\[\dfr{q, u} = \set[1]{(p, s) \in F \times \dr}{p + r s \in \fr[\dotted{F}{r}]{q, u}}\]
for any $r \in \intoo{0}{\chg(F)}$, and in this form it is arguably easier to imagine what it represents. However, we purposefully define it in a way which makes it clear that it is independent of~$r$. This is particularly relevant because we will also need to consider the case $r = 0$, when the two definitions are \emph{not} equivalent.

\begin{example}
The two pictures below illustrate two dotted fragments at the point~$q$, marked by~$\circ$, the left one for $u = (1, 1)$, and the right one for $u = (1, 0)$. The solid points are $p + r s$ for all $(p, s) \in F \times \dr$, and in particular the black (not gray) solid points are $p + r s$ just for $(p, s) \in \dfr{q, u}$.
\begin{center}
\begin{tikzpicture}[scale = 0.7]
\def\pointsize{2.5pt}
\draw[black!50] (-2, -2) rectangle (9, 6);

\draw[-{Stealth[length=7pt]}] (0, 0) -- node [above=0.1pt, left=0.1pt] {$u$} (1, 1);
\draw (0, 0) circle (\pointsize);

\filldraw[black!20] (-1, -1) circle (\pointsize);
\filldraw[black!20] (0, -1) circle (\pointsize);
\filldraw[black!20] (1, -1) circle (\pointsize);
\filldraw[black!20] (-1, 0) circle (\pointsize);
\filldraw[black!20] (1, 0) circle (\pointsize);
\filldraw[black!20] (-1, 1) circle (\pointsize);
\filldraw[black!20] (0, 1) circle (\pointsize);
\filldraw (1, 1) circle (\pointsize);

\filldraw (3, 3) circle (\pointsize);
\filldraw (4, 3) circle (\pointsize);
\filldraw (5, 3) circle (\pointsize);
\filldraw (3, 4) circle (\pointsize);
\filldraw (5, 4) circle (\pointsize);
\filldraw (3, 5) circle (\pointsize);
\filldraw (4, 5) circle (\pointsize);
\filldraw (5, 5) circle (\pointsize);

\filldraw[black!20] (6, -1) circle (\pointsize);
\filldraw[black!20] (7, -1) circle (\pointsize);
\filldraw[black!20] (8, -1) circle (\pointsize);
\filldraw[black!20] (6, 0) circle (\pointsize);
\filldraw[black!20] (8, 0) circle (\pointsize);
\filldraw (6, 1) circle (\pointsize);
\filldraw (7, 1) circle (\pointsize);
\filldraw (8, 1) circle (\pointsize);
\end{tikzpicture}
\quad
\begin{tikzpicture}[scale = 0.7]
\def\pointsize{2.5pt}
\draw[black!50] (-2, -2) rectangle (9, 6);

\draw[-{Stealth[length=7pt]}] (0, 0) -- node [above=0.1pt] {$u$} (1, 0);
\draw (0, 0) circle (\pointsize);

\filldraw[black!20] (-1, -1) circle (\pointsize);
\filldraw[black!20] (0, -1) circle (\pointsize);
\filldraw[black!20] (1, -1) circle (\pointsize);
\filldraw[black!20] (-1, 0) circle (\pointsize);
\filldraw (1, 0) circle (\pointsize);
\filldraw[black!20] (-1, 1) circle (\pointsize);
\filldraw[black!20] (0, 1) circle (\pointsize);
\filldraw[black!20] (1, 1) circle (\pointsize);

\filldraw[black!20] (3, 3) circle (\pointsize);
\filldraw[black!20] (4, 3) circle (\pointsize);
\filldraw[black!20] (5, 3) circle (\pointsize);
\filldraw[black!20] (3, 4) circle (\pointsize);
\filldraw[black!20] (5, 4) circle (\pointsize);
\filldraw[black!20] (3, 5) circle (\pointsize);
\filldraw[black!20] (4, 5) circle (\pointsize);
\filldraw[black!20] (5, 5) circle (\pointsize);

\filldraw[black!20] (6, -1) circle (\pointsize);
\filldraw[black!20] (7, -1) circle (\pointsize);
\filldraw[black!20] (8, -1) circle (\pointsize);
\filldraw (6, 0) circle (\pointsize);
\filldraw (8, 0) circle (\pointsize);
\filldraw[black!20] (6, 1) circle (\pointsize);
\filldraw[black!20] (7, 1) circle (\pointsize);
\filldraw[black!20] (8, 1) circle (\pointsize);
\end{tikzpicture}
\end{center}
\end{example}

\begin{lemma}\label{lemma:dotted-fragment-cases}
The following holds for all $N \in \NN$, $F \in \ifsub(\ell_1^N)$, $q \in F$, $u \in \dr$, $(p, s) \in \dfr{q, u}$, and $k \in \intcc[\NN]{1}{N}$.
\begin{enumerate}
\item
If $p_k = q_k$, then $u_k = s_k$.
\item
If $p_k > q_k$, then $u_k = 1$.
\item
If $p_k < q_k$, then $u_k = -1$.
\end{enumerate}
\end{lemma}

\begin{proof}
Set $g := \tfrac{\chg(F)}{2}$. Note that for $(p, s) \in \dfr{q, u}$ we have $\sgn(p_k + g s_k - q_k) = u_k$.
\begin{enumerate}
\item
If $p_k = q_k$, then $u_k = \sgn(p_k + g s_k - q_k) = \sgn(g s_k) = \sgn(s_k) = s_k$.
\item
If $p_k > q_k$, then $p_k - g > q_k + g$ by Lemma~\ref{lemma:coordinate-half-gap}. Hence $p_k + g s_k - q_k > 0$, so $u_k = \sgn(p_k + g s_k - q_k) = 1$.
\item
Analogous to the previous item.
\end{enumerate}
\end{proof}

\begin{lemma}\label{lemma:dotted-fragments}
Let $N \in \NN$ and $F \in \ifsub(\ell_1^N)$.
\begin{enumerate}
\item\label{lemma:dotted-fragments:partition}
For any $q \in F$, the dotted fragments $\set{\dfr{q, u}}{u \in \dr}$ form a partition of $F \times \dr$. That is, $F \times \dr = \bigcup_{u \in \dr} \dfr{q, u}$, and this is a disjoint union.
\item\label{lemma:dotted-fragments:projection-formula}
For all $q \in F$, $u \in \dr$, $(p, s) \in \dfr{q, u}$, $t \in \intcc{-1}{1}^N$ and $r \in \intco{0}{\chg(F)}$, we have
\[d_1\big(\face_{p, s}(r), q + r u\big) + d_1\big(\face_{q, u}(r), q + r t\big) = d_1\big(\face_{p, s}(r), q + r t\big).\]
\end{enumerate}
\end{lemma}

\begin{proof}
\
\begin{enumerate}
\item
Follows from the definition of the dotted fragment and because the definition of the coordinate half-gap implies $p + \tfrac{\chg(F)}{2} s \neq q$.
\item
Using Proposition~\ref{proposition:face-projection} and Lemma~\ref{lemma:coordinate-half-gap}, we may calculate
\begin{gather*}
d_1\big(\face_{p, s}(r), q + r t\big) = d_1\big(\proj_{p, s, r}(q + r t), q + r t\big) \\
= \sum_{k \in \intcc[\NN]{1}{N}} \Big|\big(\proj_{p, s, r}(q + r t)\big)_k - (q_k + r t_k)\Big| = \sum_{k \in \intcc[\NN]{1}{N}} \begin{cases} r (|s_k| - s_k t_k) & \text{if $p_k = q_k$}, \\ (p_k + r L_{s_k}) - (q_k + r t_k) & \text{if $p_k > q_k$}, \\ (q_k + r t_k) - (p_k + r R_{s_k}) & \text{if $p_k < q_k$}. \end{cases} \\
\end{gather*}
As this holds for arbitrary $t \in \intcc{-1}{1}^N$, we may take $t = u$ to get (while also using Lemma~\ref{lemma:dotted-fragment-cases})
\begin{gather*}
d_1\big(\face_{p, s}(r), q + r u\big) = \sum_{k \in \intcc[\NN]{1}{N}} \begin{cases} r (|s_k| - s_k u_k) & \text{if $p_k = q_k$}, \\ (p_k + r L_{s_k}) - (q_k + r u_k) & \text{if $p_k > q_k$}, \\ (q_k + r u_k) - (p_k + r R_{s_k}) & \text{if $p_k < q_k$} \end{cases} \\
= \sum_{k \in \intcc[\NN]{1}{N}} \begin{cases} 0 & \text{if $p_k = q_k$}, \\ (p_k + r L_{s_k}) - (q_k + r) & \text{if $p_k > q_k$}, \\ (q_k - r) - (p_k + r R_{s_k}) & \text{if $p_k < q_k$}. \end{cases}
\end{gather*}
Proposition~\ref{proposition:face-projection-from-cube-itself} gives us $d_1\big(\face_{q, u}(r), q + r t\big) = \sum_{k \in \intcc[\NN]{1}{N}} r \big(|u_k| - u_k t_k\big)$. Putting all this together and applying Lemma~\ref{lemma:dotted-fragment-cases}, we conclude
\begin{gather*}
d_1\big(\face_{p, s}(r), q + r u\big) + d_1\big(\face_{q, u}(r), q + r t\big) \\
= \left(\sum_{k \in \intcc[\NN]{1}{N}} \begin{cases} 0 & \text{if $p_k = q_k$}, \\ (p_k + r L_{s_k}) - (q_k + r) & \text{if $p_k > q_k$}, \\ (q_k - r) - (p_k + r R_{s_k}) & \text{if $p_k < q_k$} \end{cases}\right) + \sum_{k \in \intcc[\NN]{1}{N}} r \big(|u_k| - u_k t_k\big) \\
= \sum_{k \in \intcc[\NN]{1}{N}} \begin{cases} r (|s_k| - s_k t_k) & \text{if $p_k = q_k$}, \\ (p_k + r L_{s_k}) - (q_k + r) + r (1 - t_k) & \text{if $p_k > q_k$}, \\ (q_k - r) - (p_k + r R_{s_k}) + r (1 + t_k) & \text{if $p_k < q_k$} \end{cases} \\
= \sum_{k \in \intcc[\NN]{1}{N}} \begin{cases} r (|s_k| - s_k t_k) & \text{if $p_k = q_k$}, \\ (p_k + r L_{s_k}) - (q_k + r t_k) & \text{if $p_k > q_k$}, \\ (q_k + r t_k) - (p_k + r R_{s_k}) & \text{if $p_k < q_k$} \end{cases} \\
= d_1\big(\face_{p, s}(r), q + r t\big).
\end{gather*}
\end{enumerate}
\end{proof}

\section{Fragment System}\label{section:fragment-system}

As Theorem~\ref{theorem:weight-measure-of-union-of-cubes} below demonstrates, the weight measure of a (sufficiently nice) union of cubes can be given as a linear combination of the weight measures of the faces of the cubes, and the coefficients in this linear combinations are given in terms of solutions of a certain system of linear equations, dubbed `Fragment System'. This section is dedicated to the study of this system.

The strategy here is a more general version of the strategy in~\cite{kalisnik2026continuitymagnitudeskewfinite} where we likewise gave a formula for the weight measure of certain unions of cubes. Specifically, the following definition of the Fragment System was inspired by and generalizes the `Corner System'~\cite[Definition~5.8]{kalisnik2026continuitymagnitudeskewfinite}.

\begin{definition}\label{definition:fragment-system}
For any $N \in \NN$, $F \in \ifsub(\ell_1^N)$ and $r \in \intco{0}{\chg(F)}$, we define the following system of linear equations for unknowns~$(x_{p, s})_{(p, s) \in F \times \dr}$.
\begin{samepage}
\begin{center}
\underline{Fragment~System}:
\end{center}
\[\Bigg(\sum_{(p, s) \in \dfr{q, u}}\!\!\!\!x_{p, s}\,e^{-d_1(\face_{p, s}(r), q + r u)} \ \ = \ \sum_{p \in \fr{q, u}}\!\!e^{-d_1(\cc_p(r), q + r u)}\Bigg)_{(q, u) \in F \times \dr}\]
\end{samepage}

Additionally, we define $FS(r)$ to be the coefficient matrix of the Fragment System. This system has $\card{F} \cdot (3^N - 1)$ equations and the same number of unknowns.
\end{definition}

\begin{lemma}\label{lemma:fragment-system-solvable-at-zero}
Let $N \in \NN$ and $F \in \ifsub(\ell_1^N)$. Then $FS(0)$ (the coefficient matrix of the Fragment System for $r = 0$) is invertible.
\end{lemma}

\begin{proof}
Since $FS(0)$ is a square matrix, it suffices to show that its kernel is trivial. Let us assume that $x = (x_{p, s})_{(p, s) \in F \times \dr}$ is a solution of $FS(0)\,x = 0_{F \times \dr}$. From this we will derive that $x = 0_{F \times \dr}$.

For every $p \in F$ and $u \in \set{-1, 0, 1}^N$ denote
\[y_{p, u} \ := \sum_{s \in \sub_u} x_{p, s}.\]
We claim that all $y_{p, u}$ are~$0$. We will prove this by induction on~$\fdim_u$, starting at the top dimension and going down.

By assumption we have
\[\sum_{(p, s) \in \dfr{q, u}}\!\!\!\!x_{p, s}\,e^{-d_1(p, q)} = 0\]
for every $q \in F$ and $u \in \dr$. Sum these equations for a fixed~$q$ over all~$u$. We get
\[0 =\!\!\sum_{u \in \dr} \sum_{(p, s) \in \dfr{q, u}}\!\!\!\!x_{p, s}\,e^{-d_1(p, q)} = \sum_{\substack{p \in F \\ s \in \dr}}\!\!x_{p, s}\,e^{-d_1(p, q)} = \sum_{p \in F} y_{p, 0}\,e^{-d_1(p, q)}.\]
In matrix form, that is $Z_F \cdot \big(y_{p, 0}\big)_{p \in F} = 0_F$. Since the similarity matrix~$Z_F$ is positive definite, therefore invertible, we conclude $y_{p, 0} = 0$ for all $p \in F$.

Take now any $q \in F$ and $u \in \dr$, and inductively assume that we already know $y_{p, t} = 0$ for all $p \in F$ and $t \in \set{-1, 0, 1}^N$ such that $\fdim_t > \fdim_u$. Consider the following sum:
\[z := \sum_{t \in \sub_u} \sum_{(p, s) \in \dfr{q, t}}\!\!\!\!x_{p, s}\,e^{-d_1(p, q)}.\]
On one hand, $z = 0$ since by assumption $x$ is a solution of the system $FS(0)\,x = 0_{F \times \dr}$. On the other hand, since dotted fragments are disjoint, $z$ is the same as the sum of $x_{p, s}\,e^{-d_1(p, q)}$ for all $(p, s) \in F \times \dr$ such that there exists $t \in \sub_u$ so that $(p, s) \in \dfr{q, t}$. But we may calculate
\begin{align*}
&\some{t \in \sub_u}{(p, s) \in \dfr{q, t}} \\
&\iff \some{t \in I_u}\all{k \in \intcc[\NN]{1}{N}}{\sgn\big(p_k + \tfrac{\chg(F)}{2} s_k - q_k\big) = t_k} \\
&\iff \all{k \in \intcc[\NN]{1}{N}}{u_k \neq 0 \impl \sgn\big(p_k + \tfrac{\chg(F)}{2} s_k - q_k\big) = u_k} \\
&\iff \all{k \in \intcc[\NN]{1}{N}}{u_k = 0 \lor \sgn(p_k - q_k) = u_k \lor \big(p_k = q_k \land s_k = u_k\big).}
\end{align*}
Denote $\widetilde{F}_{q, u} := \set{p \in F}{\all{k \in \intcc[\NN]{1}{N}}{u_k = 0 \lor \sgn(p_k - q_k) = u_k \lor p_k = q_k}}$; then
\[z = \sum_{p \in \widetilde{F}_{q, u}} \ \sum_{\substack{s \in \dr \\ \all{k \in \intcc[\NN]{1}{N}}{0 \neq u_k \neq \sgn(p_k - q_k) \impl s_k = u_k}}}\hspace{-12ex}x_{p, s}\,e^{-d_1(p, q)}.\]
Observe that for every $p \in \widetilde{F}_{q, u}$ and $k \in \intcc[\NN]{1}{N}$:
\begin{itemize}
\item
we either have $u_k = 0 \lor \sgn(p_k - q_k) = u_k$, in which case there is no condition on~$s_k$ in the above sum, i.e.\ $s_k$ can be any of $-1$, $0$, $1$,
\item
or we have $0 \neq u_k \neq \sgn(p_k - q_k)$, in which case there is only one option for~$s_k$, and it cannot be~$0$.
\end{itemize}
Hence, if we define
\[a(p) := \left(\begin{cases} 0 & \text{if $u_k = 0 \lor \sgn(p_k - q_k) = u_k$}, \\ u_k & \text{if $0 \neq u_k \neq \sgn(p_k - q_k)$}, \end{cases}\right)_{k \in \intcc[\NN]{1}{N}}\]
then
\[z = \sum_{p \in \widetilde{F}_{q, u}} y_{p, a(p)}\,e^{-d_1(p, q)}.\]
Note that $\fdim_{a(p)} \geq \fdim_u$, and the equality $\fdim_{a(p)} = \fdim_u$ holds if and only if $\sgn(p_k - q_k) = u_k \impl u_k = 0$ for all $k \in \intcc[\NN]{1}{N}$. Hence, if we define
\[\widehat{F}_{q, u} := \set[1]{p \in \widetilde{F}_{q, u}}{\all{k \in \intcc[\NN]{1}{N}}{\sgn(p_k - q_k) = u_k \impl u_k = 0}}\]
(which simplifies to $\widehat{F}_{q, u} = \set{p \in F}{\all{k \in \intcc[\NN]{1}{N}}{u_k \neq 0 \impl p_k = q_k}}$), then by the induction hypothesis
\[z = \sum_{p \in \widehat{F}_{q, u}} y_{p, a(p)}\,e^{-d_1(p, q)}.\]
Recall from before that $z = 0$, and we clearly have $q \in \widehat{F}_{q, u}$, so we have obtained a homogeneous system of linear equations whose coefficient matrix is the similarity matrix of~$\widehat{F}_{q, u}$. As $\widehat{F}_{q, u} \subseteq F$, this matrix is positive definite, therefore invertible, so $y_{p, a(p)} = 0$ for all $p \in \widehat{F}_{q, u}$. In particular, we have $a(q) = u$, so $y_{q, u} = 0$.

The intuition behind this whole argument is as follows: the sum~$z$ is a linear combination of some~$y_{p, t}$ (which include $y_{q, u}$) with $\fdim_t \geq \fdim_u$, and we have $\fdim_t = \fdim_u$ precisely when $\face_{p, t}(\tfrac{\chg(F)}{2})$ is contained in the affine hull of $\face_{q, u}(\tfrac{\chg(F)}{2})$. We conclude that all these $y_{p, t}$ are zero since they solve a homogeneous system whose coefficient matrix is the similarity matrix of the centers of the faces $\face_{p, t}(\tfrac{\chg(F)}{2})$ in the aforementioned affine hull.

We now show that all $x_{p, s}$ are~$0$ by induction on~$\fdim_s$, this time in increasing order. If $\fdim_s = 0$, then $x_{p, s} = y_{p, s} = 0$.

Take any $q \in F$ and $u \in \dr$, and inductively assume that $x_{p, s} = 0$ for all $p \in F$ and $s \in \dr$ such that $\fdim_s < \fdim_u$. Then
\[x_{q, u} \ = \ y_{q, u} - \!\!\!\!\!\!\!\!\sum_{s \in \sub_u \setminus \set{u, 0_N}}\!\!\!\!\!\!\!\!\!\!x_{p, s} \ = \ 0 - \!\!\!\!\!\!\!\!\sum_{s \in \sub_u \setminus \set{u, 0_N}}\!\!\!\!\!\!\!\!\!\!0 \ = \ 0.\]
\end{proof}

\begin{lemma}\label{lemma:system-solvability}
The following holds for any $N \in \NN$ and $F \in \ifsub(\ell_1^N)$.
\begin{enumerate}
\item
There exists $\rho \in \intoc{0}{\chg(F)}$ such that for every $r \in \intco{0}{\rho}$ the Fragment System has a unique solution.
\item
Let $w = (w_p)_{p \in F}$ be the weighting of~$F$, $(x_{p, s})_{p \in F, s \in \dr}$ the solution of the Fragment System for $r = 0$, and $y_p := \sum_{s \in \dr} x_{p, s}$. Then $y_p = 1 - w_p$ for all $p \in F$.
\end{enumerate}
\end{lemma}

\begin{proof}
\
\begin{enumerate}
\item
Clearly $FS\colon \intco{0}{\chg(F)} \to \RR^{(F \times \dr) \times (F \times \dr)} \ism \RR^{\card{F} \cdot (3^N - 1) \times \card{F} \cdot (3^N - 1)}$ is continuous, and by Lemma~\ref{lemma:fragment-system-solvable-at-zero} $\det{FS(0)} \neq 0$. Hence there exists $\rho \in \intoc{0}{\chg(F)}$ such that $\det{FS(r)} \neq 0$ for all $r \in \intco{0}{\rho}$. Therefore the Fragment System is uniquely solvable for these~$r$.
\item
By assumption we have
\[\sum_{(p, s) \in \dfr{q, u}}\!\!\!\!x_{p, s}\,e^{-d_1(p, q)} = \sum_{p \in \fr{q, u}}\!\!e^{-d_1(p, q)}\]
for all $q \in F$, $u \in \dr$. Summing these equations over~$u$ at a fixed~$q$ and using Lemma~\ref{lemma:dotted-fragments}~\ref{lemma:dotted-fragments:partition} gives us
\[\sum_{u \in \dr} \sum_{(p, s) \in \dfr{q, u}}\!\!\!\!x_{p, s}\,e^{-d_1(p, q)} = \sum_{\substack{p \in F \\ s \in \dr}}\!\!x_{p, s}\,e^{-d_1(p, q)} = \sum_{p \in F} y_p\,e^{-d_1(p, q)}\]
on the left-hand side, and (using Lemma~\ref{lemma:fragments}~\ref{lemma:fragments:partition})
\begin{gather*}
\sum_{u \in \dr} \sum_{p \in \fr{q, u}}\!\!e^{-d_1(p, q)} = \sum_{p \in F \setminus \set{q}}\!\!e^{-d_1(p, q)} = \Big(\sum_{p \in F} e^{-d_1(p, q)}\Big) - 1 \\
= \Big(\sum_{p \in F} e^{-d_1(p, q)}\Big) - \sum_{p \in F} e^{-d_1(p, q)} w_p = \sum_{p \in F} e^{-d_1(p, q)} (1 - w_p)
\end{gather*}
on the right-hand side. We can rewrite this in matrix form as
\[Z_F \cdot y = Z_F \cdot (1_F - w),\]
and since the similarity matrix~$Z_F$ is positive definite, therefore invertible, we may cancel it to obtain $y = 1_F - w$.
\end{enumerate}
\end{proof}

\section{Weight Measure Formula and Magnitude Continuity}

In this section we present the two main results of the paper: the formulas for the weight measure and magnitude of a union of cubes, and the continuity of magnitude at all finite subsets in~$\ell_1^N$.

\begin{theorem}\label{theorem:weight-measure-of-union-of-cubes}
Let $N \in \NN$ and $F \in \ifsub(\ell_1^N)$. Define
\[S_F := \set{r \in \intco{0}{\chg(F)}}{\det{FS(r)} \neq 0}.\]
For every $r \in S_F$, let $\big(\alpha_{p, s}(r)\big)_{(p, s) \in F \times \dr}$ be the unique solution of the Fragment System. Then, for every $r \in (S_F)_{> 0}$,
\[\omega_{\cc_F(r)} \ = \ \Big(\sum_{p \in F} \omega_{\cc_p(r)}\Big) \ \ - \sum_{\substack{p \in F \\ s \in \dr}}\!\!\alpha_{p, s}(r)\,\omega_{\face_{p, s}(r)}\]
is the (unique, by Proposition~\ref{proposition:uniqueness-of-weight-measure}) weight measure for the union of cubes $\cc_F(r)$. It follows that
\[\mg\big(\cc_F(r)\big) \ = \ \card{F} \cdot (1 + r)^N - \!\!\sum_{\substack{p \in F \\ s \in \dr}}\!\!\alpha_{p, s}(r) \cdot (1 + r)^{\fdim_s}.\]
\end{theorem}

\begin{proof}
To prove that the given $\omega_{\cc_F(r)}$ is a weight measure for $\cc_F(r)$, we need to show for all $q \in F$, $t \in \intcc{-1}{1}^N$ that $\int_{\cc_F(r)} e^{-d_1(x, q + r t)}\,d\omega_{\cc_F(r)}(x) = 1$. This is demonstrated by the following calculation.
\begin{align*}
&\int_{\cc_F(r)} e^{-d_1(x, q + r t)}\,d\omega_{\cc_F(r)}(x) \\
&\quad\quad\text{(integral can be restricted to the support of the integrated measure)} \\
&= \Big(\sum_{p \in F} \int_{\cc_p(r)} e^{-d_1(x, q + r t)}\,d\omega_{\cc_p(r)}(x)\Big) - \sum_{\substack{p \in F \\ s \in \dr}}\!\!\bigg(\alpha_{p, s}(r)\!\!\!\!\int\displaylimits_{\face_{p, s}(r)}\!\!\!\!e^{-d_1(x, q + r t)}\,d\omega_{\face_{p, s}(r)}(x)\bigg) \\
&\quad\quad\text{(Lemma~\ref{lemma:face-weight-measure-integration})} \\
&= \Big(\sum_{p \in F} e^{-d_1(\cc_p(r), q + r t)}\Big) - \sum_{\substack{p \in F \\ s \in \dr}}\!\!\alpha_{p, s}(r)\,e^{-d_1(\face_{p, s}(r), q + r t)} \\
&\quad\quad\text{(Lemma~\ref{lemma:dotted-fragments}~\ref{lemma:dotted-fragments:partition})} \\
&= \Big(\sum_{p \in F} e^{-d_1(\cc_p(r), q + r t)}\Big) - \sum_{u \in \dr} \sum_{(p, s) \in \dfr{q, u}}\!\!\!\!\!\!\alpha_{p, s}(r)\,e^{-d_1(\face_{p, s}(r), q + r t)} \\
&\quad\quad\text{(Lemma~\ref{lemma:dotted-fragments}~\ref{lemma:dotted-fragments:projection-formula})} \\
&= \Big(\sum_{p \in F} e^{-d_1(\cc_p(r), q + r t)}\Big) - \sum_{u \in \dr} \bigg(e^{-d_1(\face_{q, u}(r), q + r t)}\!\!\!\!\!\!\sum_{(p, s) \in \dfr{q, u}}\!\!\!\!\!\!\alpha_{p, s}(r)\,e^{-d_1(\face_{p, s}(r), q + r u)}\bigg) \\
&\quad\quad\text{($\alpha_{p, s}$ are solutions of the Fragment System)} \\
&= \Big(\sum_{p \in F} e^{-d_1(\cc_p(r), q + r t)}\Big) - \sum_{u \in \dr} \bigg(e^{-d_1(\face_{q, u}(r), q + r t)} \sum_{p \in \fr{q, u}}\!\!e^{-d_1(\cc_p(r), q + r u)}\bigg) \\
&\quad\quad\text{(Lemma~\ref{lemma:fragments}~\ref{lemma:fragments:projection-formula})} \\
&= \Big(\sum_{p \in F} e^{-d_1(\cc_p(r), q + r t)}\Big) - \sum_{u \in \dr} \sum_{p \in \fr{q, u}}\!\!e^{-d_1(\cc_p(r), q + r t)} \\
\end{align*}
\begin{align*}
&\quad\quad\text{(Lemma~\ref{lemma:fragments}~\ref{lemma:fragments:partition})} \\
&= \Big(\sum_{p \in F} e^{-d_1(\cc_p(r), q + r t)}\Big) - \sum_{p \in F \setminus \set{q}}\!\!e^{-d_1(\cc_p(r), q + r t)} \\
&\quad\quad\text{(all but one term cancel)} \\
&= e^{-d_1(\cc_q(r), q + r t)} = e^{-0} = 1
\end{align*}

We calculate magnitude from the weight measure in the usual way (recall the values of magnitude for cubes and their faces from Section~\ref{section:cubes}).
\begin{gather*}
\mg\big(\cc_F(r)\big) = \int_{\cc_F(r)} d\omega_{\cc_F(r)}(x) \\
= \Big(\sum_{p \in F} \int_{\cc_p(r)} d\omega_{\cc_p(r)}(x)\Big) - \sum_{\substack{p \in F \\ s \in \dr}}\!\!\bigg(\alpha_{p, s}(r)\!\!\!\!\int\displaylimits_{\face_{p, s}(r)}\!\!\!\!d\omega_{\face_{p, s}(r)}(x)\bigg) \\
= \Big(\sum_{p \in F} \mg\big(\cc_p(r)\big)\Big) - \sum_{\substack{p \in F \\ s \in \dr}}\!\!\alpha_{p, s}(r) \cdot \mg\big(\face_{p, s}(r)\big) \\
= \Big(\sum_{p \in F} (1 + r)^N\Big) - \sum_{\substack{p \in F \\ s \in \dr}}\!\!\alpha_{p, s}(r) \cdot (1 + r)^{\fdim_s} \\
= \card{F} \cdot (1 + r)^N - \!\!\sum_{\substack{p \in F \\ s \in \dr}}\!\!\alpha_{p, s}(r) \cdot (1 + r)^{\fdim_s}
\end{gather*}
\end{proof}

We claim that the weight of a cube around a point in a finite set tends to the weight of the point as the size of the cube tends towards~$0$, and so the weight of a union of cubes (i.e.\ its magnitude) tends towards the magnitude of the finite set.

\begin{lemma}\label{lemma:weight-limits}
Let the notation and assumptions be as in Theorem~\ref{theorem:weight-measure-of-union-of-cubes}, and additionally, let $w = (w_p)_{p \in F}$ be the weighting of~$F$.
\begin{enumerate}
\item
The maps $\alpha_{p, s}\colon S_F \to \RR$ are continuous.
\item
There exists $\rho \in \intoc{0}{\chg(F)}$ such that $\intco{0}{\rho} \subseteq S_F$.
\item
For every $p \in F$,
\[\lim_{r \decr 0} \omega_{\cc_F(r)}\big(\cc_p(r)\big) = w_p.\]
Hence
\[\lim_{r \decr 0} \mg\big(\cc_F(r)\big) = \mg(F).\]
\end{enumerate}
\end{lemma}

\begin{proof}
\
\begin{enumerate}
\item
We have $\alpha(r) = FS(r)^{-1} \cdot \big(\sum_{p \in \fr{q, u}}\!\!e^{-d_1(\cc_p(r), q + r u)}\big)_{q \in F,\,u \in \dr}$ on~$S_F$, so clearly $\alpha\colon S_F \to \RR^{F \times \dr} \ism \RR^{\card{F} \cdot (3^N - 1)}$ is continuous.
\item
By Lemma~\ref{lemma:system-solvability}.
\item
We have
\begin{gather*}
\lim_{r \decr 0} \omega_{\cc_F(r)}\big(\cc_p(r)\big) = \lim_{r \decr 0} \int_{\cc_p(r)} d\omega_{\cc_F(r)}(x) \\
= \lim_{r \decr 0} \bigg(\Big(\int_{\cc_p(r)} d\omega_{\cc_p(r)}(x)\Big) - \sum_{s \in \dr}\!\!\Big(\alpha_{p, s}(r)\!\!\!\!\int\displaylimits_{\face_{p, s}(r)}\!\!\!\!d\omega_{\face_{p, s}(r)}(x)\Big)\bigg) \\
= \lim_{r \decr 0} \bigg(\mg\big(\cc_p(r)\big) - \sum_{s \in \dr}\!\!\alpha_{p, s}(r) \cdot \mg\big(\face_{p, s}(r)\big)\bigg) \\
= \lim_{r \decr 0} \bigg((1 + r)^N - \sum_{s \in \dr}\!\!\alpha_{p, s}(r) \cdot (1 + r)^{\fdim_s}\bigg) \\
= 1 - \!\!\sum_{s \in \dr}\!\!\alpha_{p, s}(0) = 1 - (1 - w_p) = w_p
\end{gather*}
where the penultimate equality holds by Lemma~\ref{lemma:system-solvability}.

Consequently
\begin{gather*}
\lim_{r \decr 0} \mg\big(\cc_F(r)\big) = \lim_{r \decr 0} \omega_{\cc_F(r)}\big(\cc_F(r)\big) = \lim_{r \decr 0} \sum_{p \in F} \omega_{\cc_F(r)}\big(\cc_p(r)\big) \\
= \sum_{p \in F} \lim_{r \decr 0} \omega_{\cc_F(r)}\big(\cc_p(r)\big) = \sum_{p \in F} w_p = \mg(F).
\end{gather*}
\end{enumerate}
\end{proof}

This is enough to conclude continuity of magnitude for finite subsets of~$\ell_1^N$.

\begin{theorem}\label{theorem:finite-magnitude-continuity}
For every $N \in \NN$, the map $\mg[\icsub(\ell_1^N)]$ is continuous at every $F \in \ifsub(\ell_1^N)$. In particular, the map $\mg[\ifsub(\ell_1^N)]$ is continuous.
\end{theorem}

\begin{proof}
The claim follows from~\cite[Corollary~4.3]{kalisnik2026continuitymagnitudeskewfinite} since we have $\lim_{r \decr 0} \mg\big(\cc_F(r)\big) = \mg(F)$ by Lemma~\ref{lemma:weight-limits}.
\end{proof}

\section{Examples}

In this section we illustrate our theory on a few examples. But first, let us make an observation which makes solving the Fragment System quicker.

It follows easily from the definition of dotted fragments that $(q, u) \in \dfr{q, u}$ for all $(q, u) \in F \times \dr$ (where $F \in \ifsub(\ell_1^N)$). Call a dotted fragment \df{trivial} when this is its only point, i.e.\ when $\dfr{q, u} = \set{(q, u)}$.

\begin{proposition}\label{proposition:trivial-dotted-fragments}
Let $N \in \NN$ and $F \in \ifsub(\ell_1^N)$. Let $q \in F$ and $u \in \dr$ be such that the dotted fragment $\dfr{q, u}$ is trivial. Then the following holds.
\begin{enumerate}
\item
$\fr{q, u} = \emptyset$.
\item
$\alpha_{q, u}(r) = 0$ for every $r \in S_F$ (where $\alpha_{q, u}(r)$ is a coefficient in the formula for the weight measure $\omega_{\cc_F(r)}$, as given by Theorem~\ref{theorem:weight-measure-of-union-of-cubes}).
\end{enumerate}
\end{proposition}

\begin{proof}
\
\begin{enumerate}
\item
Follows from the facts that $q \notin \fr{q, u}$ and if $p \in \fr{q, u}$, then $(p, u) \in \dfr{q, u}$.
\item
Because under the given assumptions, one of the equations in the Fragment System ends up being $x_{q, u} = 0$.
\end{enumerate}
\end{proof}

In practice many dotted fragments are trivial, so we can quickly solve a good portion of the Fragment System.

\begin{example}\label{example:pair}
Take $a \in \RR_{> 0}$ and let $F := \set{(0, 0), (a, 0)} \subseteq \ell_1^2$. Then $\chg(F) = \tfrac{a}{2}$. Let $r \in \intoo{0}{\chg(F)}$. Using Theorem~\ref{theorem:weight-measure-of-union-of-cubes}, we could determine $\omega_{\cc_F(r)}$ by solving a linear system with $16$~equations and unknowns, but we can streamline this significantly.

Most dotted fragments are trivial; Proposition~\ref{proposition:trivial-dotted-fragments} gives us $\alpha_{(0, 0), (-1, -1)} = \alpha_{(0, 0), (0, -1)} = \alpha_{(0, 0), (-1, 0)} = \alpha_{(0, 0), (-1, 1)} = \alpha_{(0, 0), (0, 1)} = \alpha_{(a, 0), (0, -1)} = \alpha_{(a, 0), (1, -1)} = \alpha_{(a, 0), (1, 0)} = \alpha_{(a, 0), (0, 1)} = \alpha_{(a, 0), (1, 1)} = 0$. Due to symmetry, we must have $\alpha_{(0, 0), (1, -1)} = \alpha_{(0, 0), (1, 1)} = \alpha_{(a, 0), (-1, -1)} = \alpha_{(a, 0), (-1, 1)} =: \alpha$ and $\alpha_{(0, 0), (1, 0)} = \alpha_{(a, 0), (-1, 0)} =: \beta$. The following picture illustrates these coefficients.
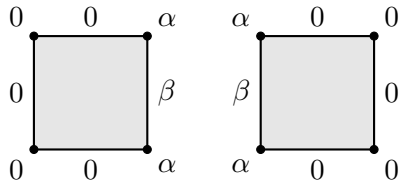
\begin{center}
\begin{tikzpicture}[scale = 3]
\def\r{0.25}
\def\pointsize{0.5pt}
\filldraw[fill = black!10, thick] (-\r, -\r) rectangle (\r, \r);
\filldraw[fill = black!10, thick] (1-\r, -\r) rectangle (1+\r, \r);
\filldraw (-\r, -\r) circle (\pointsize) node [below left] {$0$};
\filldraw (\r, -\r) circle (\pointsize) node [below right] {$\alpha$};
\filldraw (-\r, \r) circle (\pointsize) node [above left] {$0$};
\filldraw (\r, \r) circle (\pointsize) node [above right] {$\alpha$};
\draw (-\r, 0) node [left] {$0$};
\draw (\r, 0) node [right] {$\beta$};
\draw (0, -\r) node [below] {$0$};
\draw (0, \r) node [above] {$0$};
\filldraw (1-\r, -\r) circle (\pointsize) node [below left] {$\alpha$};
\filldraw (1-\r, \r) circle (\pointsize) node [above left] {$\alpha$};
\draw (1-\r, 0) node [left] {$\beta$};
\filldraw (1+\r, -\r) circle (\pointsize) node [below right] {$0$};
\filldraw (1+\r, \r) circle (\pointsize) node [above right] {$0$};
\draw (1+\r, 0) node [right] {$0$};
\draw (1, -\r) node [below] {$0$};
\draw (1, \r) node [above] {$0$};
\end{tikzpicture}
\end{center}
The Fragment System equation, indexed by $((0, 0), (1, 1))$, then reduces to $\alpha + e^{-(a - 2r)} \alpha = 0$, so $\alpha = 0$. This leaves the equation, indexed by $((0, 0), (1, 0))$, which simplifies to $\beta + e^{-(a - 2r)} \beta = e^{-(a - 2r)}$, so $\beta = \tfrac{e^{-(a - 2r)}}{1 + e^{-(a - 2r)}}$. Hence
\[\omega_{\cc_F(r)} = \omega_{\cc_{(0, 0)}(r)} + \omega_{\cc_{(a, 0)}(r)} - \tfrac{e^{-(a - 2r)}}{1 + e^{-(a - 2r)}} \Big(\omega_{\face_{(0, 0), (1, 0)}(r)} + \omega_{\face_{(a, 0), (-1, 0)}(r)}\Big)\]
and
\[\mg(\cc_F(r)) = 2 (1 + r)^2 - \tfrac{e^{-(a - 2r)}}{1 + e^{-(a - 2r)}} 2 (1 + r).\]
However, this is a simple enough example, that we can calculate this also via more elementary methods. Note that the given $\cc_F(r)$ is a $1$-product of two sets, and the magnitude (resp.\ the weight measure) of a $1$-product is the product of magnitudes (resp.\ weight measures) of factors~(see~\cite[Proposition 5.3.7.]{LM17}). Using this and recalling the magnitude of a (union of) interval(s)~(see~\cite[Corollary~5.4.3]{LM17}), we calculate
\begin{gather*}
\mg(\cc_F(r)) = \mg\big((\intcc{-r}{r} \cup \intcc{a - r}{a + r}) \times_1 \intcc{-r}{r}\big) = \mg\big(\intcc{-r}{r} \cup \intcc{a - r}{a + r}\big) \cdot \mg\big(\intcc{-r}{r}\big) \\
= \big(1 + 2r + \tanh(\tfrac{a - 2r}{2})\big) \cdot \big(1 + r\big) = \big(1 + 2r + \tfrac{1 - e^{-(a - 2r)}}{1 + e^{-(a - 2r)}}\big) \cdot \big(1 + r\big) = \big(2 + 2r - \tfrac{2 e^{-(a - 2r)}}{1 + e^{-(a - 2r)}}\big) \cdot \big(1 + r\big);
\end{gather*}
we see that we got the same result.

We can also verify the correctness of the above formula for the weight measure by rewriting it in terms of Lebesgue measures. Calculate:
\begin{align*}
\omega_{\cc_F(r)} &= \tfrac{1}{4} \Big(\lambda_{\intcc{-r}{r} \times \intcc{-r}{r}}^2 + \lambda_{\intcc{-r}{r} \times \set{-r}}^1 + \lambda_{\intcc{-r}{r} \times \set{r}}^1 + \lambda_{\set{-r} \times \intcc{-r}{r}}^1 + \lambda_{\set{r} \times \intcc{-r}{r}}^1 \\
&+ \delta_{(-r, -r)} + \delta_{(r, -r)} + \delta_{(-r, r)} + \delta_{(r, r)}\Big) \\
&+ \tfrac{1}{4} \Big(\lambda_{\intcc{a - r}{a + r} \times \intcc{-r}{r}}^2 + \lambda_{\intcc{a - r}{a + r} \times \set{-r}}^1 + \lambda_{\intcc{a - r}{a + r} \times \set{r}}^1 + \lambda_{\set{a - r} \times \intcc{-r}{r}}^1 + \lambda_{\set{a + r} \times \intcc{-r}{r}}^1 \\
&+ \delta_{(a - r, -r)} + \delta_{(a + r, -r)} + \delta_{(a - r, r)} + \delta_{(a + r, r)}\Big) \\
&- \tfrac{e^{-(a - 2r)}}{1 + e^{-(a - 2r)}} \tfrac{1}{2} \Big(\lambda_{\set{r} \times \intcc{-r}{r}}^1 + \delta_{(r, -r)} + \delta_{(r, r)} + \lambda_{\set{a - r} \times \intcc{-r}{r}}^1 + \delta_{(a - r, -r)} + \delta_{(a - r, r)}\Big) \\
&= \tfrac{1}{4} \Big(\lambda_{\intcc{-r}{r} \times \intcc{-r}{r}}^2 + \lambda_{\intcc{a - r}{a + r} \times \intcc{-r}{r}}^2 \\
&+ \lambda_{\intcc{-r}{r} \times \set{-r}}^1 + \lambda_{\intcc{-r}{r} \times \set{r}}^1 + \lambda_{\set{-r} \times \intcc{-r}{r}}^1 + \lambda_{\intcc{a - r}{a + r} \times \set{-r}}^1 + \lambda_{\intcc{a - r}{a + r} \times \set{r}}^1 + \lambda_{\set{a + r} \times \intcc{-r}{r}}^1 \\
&+ \delta_{(-r, -r)} + \delta_{(-r, r)} + \delta_{(a + r, -r)} + \delta_{(a + r, r)}\Big) \\
&+ \tfrac{1}{4} \tfrac{1 - e^{-(a - 2r)}}{1 + e^{-(a - 2r)}} \Big(\lambda_{\set{r} \times \intcc{-r}{r}}^1 + \lambda_{\set{a - r} \times \intcc{-r}{r}}^1 + \delta_{(r, -r)} + \delta_{(r, r)} + \delta_{(a - r, -r)} + \delta_{(a - r, r)}\Big).
\end{align*}
On the other hand,
\begin{align*}
\omega_{\intcc{-r}{r} \cup \intcc{a - r}{a + r}} \times \omega_{\intcc{-r}{r}} &= \Big(\tfrac{1}{2} \big(\lambda_{\intcc{-r}{r}}^1 + \lambda_{\intcc{a - r}{a + r}}^1 + \delta_{-r} + \delta_{a + r}\big) + \tfrac{1}{2} \tanh(\tfrac{a - 2r}{2}) \big(\delta_{r} + \delta_{a - r}\big)\Big) \\
&\cdot \tfrac{1}{2}\Big(\lambda_{\intcc{-r}{r}}^1 + \delta_{-r} + \delta_{r}\Big) \\
&= \tfrac{1}{4} \Big(\lambda_{\intcc{-r}{r}}^1 + \lambda_{\intcc{a - r}{a + r}}^1 + \delta_{-r} + \delta_{a + r} + \tfrac{1 - e^{-(a - 2r)}}{1 + e^{-(a - 2r)}} \big(\delta_{r} + \delta_{a - r}\big)\Big) \\
&\cdot \Big(\lambda_{\intcc{-r}{r}}^1 + \delta_{-r} + \delta_{r}\Big).
\end{align*}
We see that we again get the same result.

In any case, we have
\[\lim_{r \decr 0} \mg\big(\cc_F(r)\big) = \lim_{r \decr 0} \Big(2 (1 + r)^2 - \tfrac{e^{-(a - 2r)}}{1 + e^{-(a - 2r)}} 2 (1 + r)\Big) = 2 - \tfrac{2 e^{-a}}{1 + e^{-a}} = \tfrac{2}{1 + e^{-a}}.\]
This matches Lemma~\ref{lemma:weight-limits} since indeed by~\cite[Example 2.1.1 ii)] {leinster2010} $\mg\big(\set{(0, 0), (a, 0)}\big) = \tfrac{2}{1 + e^{-a}}$.
\end{example}

\begin{example}
Observe that if $F \in \ifsub(\ell_1^N)$ is `skew' in the sense of~\cite[Definition~4.4]{kalisnik2026continuitymagnitudeskewfinite}, then for every $(q, u) \in F \times \dr$, if $\fdim_u > 0$, then the dotted fragment $\dfr{q, u}$ is trivial. Using Proposition~\ref{proposition:trivial-dotted-fragments}, we see how the formula for the weight measure in~\cite[Theorem~5.1]{kalisnik2026continuitymagnitudeskewfinite} is a special case of the one in Theorem~\ref{theorem:weight-measure-of-union-of-cubes} above. In particular, examples in~\cite[Section~6]{kalisnik2026continuitymagnitudeskewfinite} also serve as examples of the theory, developed in this paper.
\end{example}

\begin{example}
Let $F := \set[1]{(0, 0), (1, 0), (-1, 0), (0, 1), (0, -1)} \subseteq \ell_1^2$, so $\chg(F) = \tfrac{1}{2}$. Let $r \in \intoo{0}{\chg(F)}$. The Fragment System has $40$~equations and unknowns, but taking into account trivial dotted fragments and symmetry, we can cut them down to~$4$. We will denote the four unknown coefficients by $\alpha$, $\beta$, $\gamma$, $\delta$. For visualization, we write the coefficients next to their corresponding faces in the following picture.
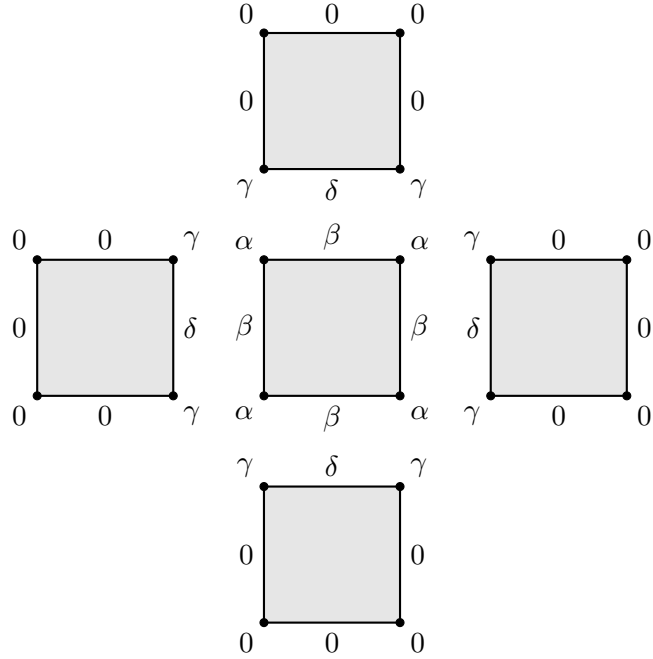
\begin{center}
\begin{tikzpicture}[scale = 3]
\def\r{0.3}
\def\pointsize{0.5pt}
\filldraw[fill = black!10, thick] (-\r, -\r) rectangle (\r, \r);
\filldraw[fill = black!10, thick] (1-\r, -\r) rectangle (1+\r, \r);
\filldraw[fill = black!10, thick] (-1-\r, -\r) rectangle (-1+\r, \r);
\filldraw[fill = black!10, thick] (-\r, 1-\r) rectangle (\r, 1+\r);
\filldraw[fill = black!10, thick] (-\r, -1-\r) rectangle (\r, -1+\r);
\filldraw (-\r, -\r) circle (\pointsize) node [below left] {$\alpha$};
\filldraw (\r, -\r) circle (\pointsize) node [below right] {$\alpha$};
\filldraw (-\r, \r) circle (\pointsize) node [above left] {$\alpha$};
\filldraw (\r, \r) circle (\pointsize) node [above right] {$\alpha$};
\draw (-\r, 0) node [left] {$\beta$};
\draw (\r, 0) node [right] {$\beta$};
\draw (0, -\r) node [below] {$\beta$};
\draw (0, \r) node [above] {$\beta$};
\filldraw (1-\r, -\r) circle (\pointsize) node [below left] {$\gamma$};
\filldraw (1-\r, \r) circle (\pointsize) node [above left] {$\gamma$};
\filldraw (-1+\r, -\r) circle (\pointsize) node [below right] {$\gamma$};
\filldraw (-1+\r, \r) circle (\pointsize) node [above right] {$\gamma$};
\filldraw (-\r, 1-\r) circle (\pointsize) node [below left] {$\gamma$};
\filldraw (\r, 1-\r) circle (\pointsize) node [below right] {$\gamma$};
\filldraw (-\r, -1+\r) circle (\pointsize) node [above left] {$\gamma$};
\filldraw (\r, -1+\r) circle (\pointsize) node [above right] {$\gamma$};
\draw (1-\r, 0) node [left] {$\delta$};
\draw (-1+\r, 0) node [right] {$\delta$};
\draw (0, 1-\r) node [below] {$\delta$};
\draw (0, -1+\r) node [above] {$\delta$};
\filldraw (1+\r, -\r) circle (\pointsize) node [below right] {$0$};
\filldraw (1+\r, \r) circle (\pointsize) node [above right] {$0$};
\filldraw (-1-\r, -\r) circle (\pointsize) node [below left] {$0$};
\filldraw (-1-\r, \r) circle (\pointsize) node [above left] {$0$};
\filldraw (-\r, 1+\r) circle (\pointsize) node [above left] {$0$};
\filldraw (\r, 1+\r) circle (\pointsize) node [above right] {$0$};
\filldraw (-\r, -1-\r) circle (\pointsize) node [below left] {$0$};
\filldraw (\r, -1-\r) circle (\pointsize) node [below right] {$0$};
\draw (1+\r, 0) node [right] {$0$};
\draw (1, -\r) node [below] {$0$};
\draw (1, \r) node [above] {$0$};
\draw (-1-\r, 0) node [left] {$0$};
\draw (-1, -\r) node [below] {$0$};
\draw (-1, \r) node [above] {$0$};
\draw (0, 1+\r) node [above] {$0$};
\draw (-\r, 1) node [left] {$0$};
\draw (\r, 1) node [right] {$0$};
\draw (0, -1-\r) node [below] {$0$};
\draw (-\r, -1) node [left] {$0$};
\draw (\r, -1) node [right] {$0$};
\end{tikzpicture}
\end{center}

Taking this into account, the Fragment System equations, indexed by $((0, 0), (1, 0))$ and $((1, 0), (-1, 0))$, reduce to
\begin{align*}
\beta + \delta e^{-(1 - 2r)} &= e^{-(1 - 2r)}, \\
\beta \big(e^{-(1 - 2r)} + e^{-1}\big) + \delta \big(1 + e^{-(2 - 2r)}\big) &= e^{-(1 - 2r)} + e^{-(2 - 2r)}.
\end{align*}
From this we get $\beta = \delta = \tfrac{e^{-(1 - 2r)}}{1 + e^{-(1 - 2r)}}$. Now take the equations, indexed by $((0, 0), (1, 1))$ and $((1, 0), (-1, 1))$; we get
\begin{align*}
\alpha + 2 \gamma e^{-(1 - 2r)} &= 0, \\
\alpha \big(e^{-(1 - 2r)} + e^{-1}\big) + \beta e^{-(1 - 2r)} + \gamma \big(1 + e^{-(2 - 2r)}\big) + \delta e^{-2(1 - 2r)} &= e^{-2(1 - 2r)}.
\end{align*}
Plugging in the already calculated $\beta$ and~$\delta$, solving the obtained system gives us $\alpha = \gamma = 0$. Hence
\begin{align*}
\omega_{\cc_F(r)} &= \omega_{\cc_{(0, 0)}(r)} + \omega_{\cc_{(1, 0)}(r)} + \omega_{\cc_{(-1, 0)}(r)} + \omega_{\cc_{(0, 1)}(r)} + \omega_{\cc_{(0, -1)}(r)} \\
&- \tfrac{e^{-(1 - 2r)}}{1 + e^{-(1 - 2r)}} \Big(\omega_{\face_{(0, 0), (1, 0)}(r)} + \omega_{\face_{(0, 0), (-1, 0)}(r)} + \omega_{\face_{(0, 0), (0, 1)}(r)} + \omega_{\face_{(0, 0), (0, -1)}(r)} \\
&+ \omega_{\face_{(1, 0), (-1, 0)}(r)} + \omega_{\face_{(-1, 0), (1, 0)}(r)} + \omega_{\face_{(0, 1), (0, -1)}(r)} + \omega_{\face_{(0, -1), (0, 1)}(r)}\Big)
\end{align*}
and
\[\mg\big(\cc_F(r)\big) = 5 (1 + r)^2 - \tfrac{e^{-(1 - 2r)}}{1 + e^{-(1 - 2r)}} 8 (1 + r).\]
Note that we have
\[\lim_{r \decr 0} \omega_{\cc_F(r)}\big(\cc_{(0, 0)}(r)\big) = \lim_{r \decr 0} \Big((1 + r)^2 - \tfrac{e^{-(1 - 2r)}}{1 + e^{-(1 - 2r)}} 4 (1 + r)\Big) = 1 - \tfrac{4 e^{-1}}{1 + e^{-1}} = \tfrac{1 - 3 e^{-1}}{1 + e^{-1}}\]
and
\[\lim_{r \decr 0} \omega_{\cc_F(r)}\big(\cc_{(1, 0)}(r)\big) = \lim_{r \decr 0} \Big((1 + r)^2 - \tfrac{e^{-(1 - 2r)}}{1 + e^{-(1 - 2r)}} (1 + r)\Big) = 1 - \tfrac{e^{-1}}{1 + e^{-1}} = \tfrac{1}{1 + e^{-1}};\]
by symmetry we get the same also for other points in $F \setminus \set{(0, 0)}$. By Lemma~\ref{lemma:weight-limits}, these should be the components of the weighting of~$F$. This we can verify directly; by symmetry, the weighting of~$F$ should be of the form $w = (x, y, y, y, y)$, and we get the equations $x + 4 y e^{-1} = 1$, $x e^{-1} + y (1 + 3 e^{-2}) = 1$. Solving this indeed gives us $x = \tfrac{1 - 3 e^{-1}}{1 + e^{-1}}$ and $y = \tfrac{1}{1 + e^{-1}}$. Note that the weight of the point $(0, 0)$ is negative; this necessitated that the coefficients belonging to the proper faces of the cube $\cc_{(0, 0)}(r)$ (at least for small~$r$) were relatively large (their average needs to be larger than one over the number of such faces).

Of course, all of this then also implies $\lim_{r \decr 0} \mg\big(\cc_F(r)\big) = \mg(F) = \tfrac{5 - 3 e^{-1}}{1 + e^{-1}}$.
\end{example}

\begin{example}
In Theorem~\ref{theorem:weight-measure-of-union-of-cubes}, we gave our formula for the weight measure under the assumption that the cube radius is smaller than the coordinate half-gap. We claim that this assumption is relevant; we do not just want the cubes to be disjoint, but also their different projections onto axes. We demonstrate that with the following example.

Take $F := \set{(0, 0), (3, 1)} \subseteq \ell_1^2$. Then $\chg(F) = \tfrac{1}{2}$, but let us take $r = 1$.
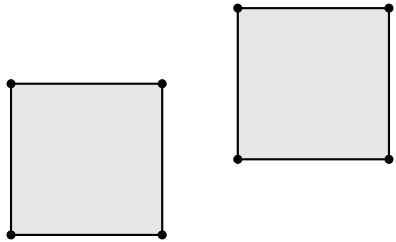
\begin{center}
\begin{tikzpicture}[scale = 1]
\def\r{1}
\def\pointsize{1.5pt}
\filldraw[fill = black!10, thick] (-\r, -\r) rectangle (\r, \r);
\filldraw[fill = black!10, thick] (3-\r, 1-\r) rectangle (3+\r, 1+\r);
\filldraw (-\r, -\r) circle (\pointsize);
\filldraw (\r, -\r) circle (\pointsize);
\filldraw (-\r, \r) circle (\pointsize);
\filldraw (\r, \r) circle (\pointsize);
\filldraw (3-\r, 1-\r) circle (\pointsize);
\filldraw (3-\r, 1+\r) circle (\pointsize);
\filldraw (3+\r, 1-\r) circle (\pointsize);
\filldraw (3+\r, 1+\r) circle (\pointsize);
\end{tikzpicture}
\end{center}
Clearly, the condition on the projections of cubes is not satisfied; the projections of the two cubes in $\cc_F(r)$ onto the $y$-axis are $\intcc{-1}{1}$ and $\intcc{0}{2}$, so they are different, but they still intersect.

We have managed to guess what the weight measure of $\cc_F(r)$ is in this case (comparing with Example~\ref{example:pair} helps); we claim
\[\omega_{\cc_F(r)} = \omega_{\intcc{-1}{1} \times \intcc{-1}{1}} + \omega_{\intcc{2}{4} \times \intcc{0}{2}} - \tfrac{e^{-1}}{1 + e^{-1}} \big(\omega_{\set{1} \times \intcc{0}{1}} + \omega_{\set{2} \times \intcc{0}{1}}\big).\]
Due to symmetry, it suffices to verify the condition for weight measure just for points in $\intcc{-1}{1} \times \intcc{-1}{1}$. First, take $q \in \intcc{-1}{1} \times \intcc{-1}{0}$. Using\footnote{Strictly speaking, $\set{1} \times \intcc{0}{1}$ and $\set{2} \times \intcc{0}{1}$ are not faces of cubes, but they are still products of compact intervals/points, and the argument is the same.} Lemma~\ref{lemma:face-weight-measure-integration}, we have
\begin{gather*}
\int_{\cc_F(r)} e^{-d_1(x, q)}\,d\omega_{\cc_F(r)}(x) \\
= e^{-d_1(\intcc{-1}{1} \times \intcc{-1}{1}, q)} + e^{-d_1(\intcc{2}{4} \times \intcc{0}{2}, q)} - \tfrac{e^{-1}}{1 + e^{-1}} \Big(e^{-d_1(\set{1} \times \intcc{0}{1}, q)} + e^{-d_1(\set{2} \times \intcc{0}{1}, q)}\Big) \\
= 1 + e^{-(2 - q_1 - q_2)} - \tfrac{e^{-1}}{1 + e^{-1}} \Big(e^{-(1 - q_1 - q_2)} + e^{-(2 - q_1 - q_2)}\Big) \\
= 1 + e^{-(2 - q_1 - q_2)} - \tfrac{e^{-1}}{1 + e^{-1}}\,e^{-(1 - q_1 - q_2)} \big(1 + e^{-1}\big) \\
= 1 + e^{-(2 - q_1 - q_2)} - e^{-(2 - q_1 - q_2)} = 1.
\end{gather*}
Similarly, for $q \in \intcc{-1}{1} \times \intcc{0}{1}$:
\begin{gather*}
\int_{\cc_F(r)} e^{-d_1(x, q)}\,d\omega_{\cc_F(r)}(x) \\
= e^{-d_1(\intcc{-1}{1} \times \intcc{-1}{1}, q)} + e^{-d_1(\intcc{2}{4} \times \intcc{0}{2}, q)} - \tfrac{e^{-1}}{1 + e^{-1}} \Big(e^{-d_1(\set{1} \times \intcc{0}{1}, q)} + e^{-d_1(\set{2} \times \intcc{0}{1}, q)}\Big) \\
= 1 + e^{-(2 - q_1)} - \tfrac{e^{-1}}{1 + e^{-1}} \Big(e^{-(1 - q_1)} + e^{-(2 - q_1)}\Big) \\
= 1 + e^{-(2 - q_1)} - \tfrac{e^{-1}}{1 + e^{-1}}\,e^{-(1 - q_1)} \big(1 + e^{-1}\big) \\
= 1 + e^{-(2 - q_1)} - e^{-(2 - q_1)} = 1.
\end{gather*}
Since the weight measure is unique (Proposition~\ref{proposition:uniqueness-of-weight-measure}), it is impossible for the given $\cc_F(r)$ to have a weight measure of the form as in Theorem~\ref{theorem:weight-measure-of-union-of-cubes}.
\end{example}

\section{Conclusion}

The following are the main results of this paper:
\begin{itemize}
\item
We gave a formula for the weight measure and magnitude of sufficiently nice unions of cubes in~$\ell_1^N$ (which includes unions of small cubes around finite subsets of~$\ell_1^N$), see Theorem~\ref{theorem:weight-measure-of-union-of-cubes}.
\item
We proved that magnitude is continuous at all finite subsets of~$\ell_1^N$, see Theorem~\ref{theorem:finite-magnitude-continuity}.
\end{itemize}
The crucial tool that we had to come up with to derive these results was the specific system of linear equations which determines the coefficients in the weight measure formula (the `Fragment System'), see Definition~\ref{definition:fragment-system}.

The weight measure formula in Theorem~\ref{theorem:weight-measure-of-union-of-cubes} is valid for all $r \in S_F \subseteq \intco{0}{\chg(F)}$, i.e.\ when the Fragment System has a unique solution. The proof of Theorem~\ref{theorem:weight-measure-of-union-of-cubes} actually tells us that every solution of the Fragment System yields a weight measure, which implies that we cannot have more than one solution, otherwise we would contradict Proposition~\ref{proposition:uniqueness-of-weight-measure}. That is, uniqueness of solutions of the Fragment System is guaranteed; but the question of existence remains. We have shown that a solution exists at least for $r$ in some neighborhood of~$0$ (Lemma~\ref{lemma:system-solvability}), and that was enough for the proof of magnitude continuity. But what happens for larger~$r$? We do not yet have an answer, and we pose this as the following question.

\begin{question}
Can we characterize the set~$S_F$, or at least give a good sufficient condition for $r \in S_F$? Is it even the case that the Fragment System is always solvable, i.e.\ $S_F = \intco{0}{\chg(F)}$?
\end{question}

We note that the Fragment System was solvable in every example we have tried thus far.


{\bf AI Statement} No AI was used to prepare this manuscript.

\bibliographystyle{plainurl}
\bibliography{magnitude}

\end{document}